\documentclass[a4paper,12pt,reqno]{amsart}

\usepackage[english]{babel}
\usepackage[utf8]{inputenc}

\usepackage[a4paper,top=2cm,bottom=2cm,left=2cm,right=2cm]{geometry}

\usepackage[title]{appendix}
\usepackage{titlesec}
\usepackage{lipsum}
\usepackage{enumerate}
\usepackage{enumitem}
\usepackage{mathrsfs}
\usepackage{amsmath,amssymb,amsthm}
\usepackage{mathtools}
\usepackage{graphicx,float,subfigure}
\usepackage{url}
\usepackage{bbold}
\usepackage{color}
\usepackage{xcolor}

\usepackage[all,cmtip]{xy}
\usepackage{multirow}
\usepackage{makecell}
\usepackage[colorinlistoftodos]{todonotes}
\usepackage[colorlinks=true, allcolors=blue]{hyperref}
\usepackage{harpoon}
\usepackage{MnSymbol}
\usepackage{cases}
\allowdisplaybreaks[4]

\DeclareMathOperator{\sgn}{sgn}
\DeclareMathOperator{\wt}{wt}

\theoremstyle{plain}
\newtheorem{theorem}{\scshape Theorem}[section]
\newtheorem{proposition}[theorem]{\scshape Proposition}
\newtheorem{lemma}[theorem]{\scshape Lemma}
\newtheorem{corollary}[theorem]{\scshape Corollary}
\newtheorem{conjecture}[theorem]{\scshape Conjecture}

\newtheorem*{assumption*}{\scshape Assumption}

\newtheorem*{claim*}{Claim}

\theoremstyle{definition}

\newtheorem{definition}[theorem]{\scshape Definition}
\newtheorem{remark}[theorem]{\scshape Remark}

\newtheorem{problem}[theorem]{\scshape Problem}

\newcommand{\floor}[1]{\lfloor #1 \rfloor}
\newcommand{\M}{\operatorname{M}}
\newcommand{\pf}{\mathsf{pf}}
\renewcommand{\det}{\mathsf{det}}
\renewcommand{\wt}{\mathsf{wt}}

\numberwithin{equation}{section}

\titleformat{\section}{\centering\bfseries}{\thesection}{1em}{\MakeUppercase}
\titleformat{\subsection}{\bfseries}{\thesubsection}{1em}{}

\begin{document}
\title{Off-diagonally symmetric domino tilings of the Aztec diamond}
\author{Yi-Lin Lee}
\address{Department of Mathematics, Indiana University, Bloomington, Indiana 47405}
\email{yillee@iu.edu}

\maketitle
\begin{abstract}
We introduce a new symmetry class of domino tilings of the Aztec diamond, called the off-diagonal symmetry class, which is motivated by the off-diagonally symmetric alternating sign matrices introduced by Kuperberg in 2002. We use the method of non-intersecting lattice paths and a modification of Stembridge's Pfaffian formula for families of non-intersecting lattice paths to enumerate our new symmetry class. The number of off-diagonally symmetric domino tilings of the Aztec diamond can be expressed as a Pfaffian of a matrix whose entries satisfy a nice and simple recurrence relation.
\end{abstract}

\section{Introduction}

\subsection{Symmetry classes}\label{sec.symclass}

Consider a finite group $G$ acting on a set of combinatorial objects $X$. Let $H$ be a subgroup of $G$. A \textit{symmetry class} is a collection of $H$-invariant objects of $X$. In enumerative combinatorics, it is quite challenging to enumerate each symmetry class of $X$ because the structure of each class varies with different subgroups $H$. These usually require different methods to enumerate them.

The study of symmetry classes of plane partitions goes back to MacMahon \cite{Mac1899}, but gained more attention in the 1970s and 80s. Stanley \cite{Stan86pp} identified ten symmetry classes of plane partitions, all the symmetry classes can be enumerated by nice product formulas. We refer the interested reader to the survey paper by Krattenthaler \cite[Section 6]{Kra15} for a modern update to Stanley's paper.

An alternating sign matrix (ASM) of order $n$ is an $n \times n$ matrix with entries $0, 1$ or $-1$ such that all row and column sums are equal to $1$ and the non-zero entries alternate in sign in each row and column. They were introduced by Mills, Robbins and Rumsey \cite{MRR83} in the early 1980s. The symmetry classes of ASMs under the action of the dihedral group of order $8$ were proposed and summarized by Stanley \cite{Stan86} and Robbins \cite{Rob91, Rob00}. There are eight symmetry classes of ASMs; five of them were fully solved, one was partially solved, while for the remaining two there are no known or conjectured formulas. We refer the interested reader to the detailed account written by Behrend, Fischer and Konvalinka \cite[Section 1.2]{BFK17}.

In the 1990s, Kuperberg's seminal paper \cite{Kup96} brought the statistical mechanical six-vertex model into the study of ASMs. Later, in \cite{Kup02}, he successfully provided a unified framework using the six-vertex model to solve some of the symmetry classes of ASMs. He also introduced several new types of ASMs, such as, off-diagonally symmetric, vertically and horizontally perverse, with U-turn sides, and combined them with the original eight symmetry classes. The enumerative results of these new types of ASMs were summarized in \cite[Section 1.2]{BFK17}.

The \textit{Aztec diamond} of order $n$, denoted by $AD(n)$, is the union of all unit squares in the region $|x|+|y| \leq n+1$, which was introduced by Elkies, Larsen, Kuperberg and Propp \cite{ELKP1,ELKP2} in the early 1990s. A domino tiling of the Aztec diamond is a covering of $AD(n)$ using dominoes without gaps or overlaps. The symmetry classes of domino tilings of $AD(n)$ under the action of the dihedral group of order $8$ have been discussed by Ciucu \cite[Section 7]{Ciucu97} and Yang \cite{Yang91}. There are five symmetry classes; the enumerations of three of them have been solved, while the other two remain open (there are no known or conjectured formulas for them).

Let $r$ and $t$ be the generators corresponding to a rotation by $90$ degrees and a reflection across the vertical diagonal of $AD(n)$, respectively. Let $G= \langle r,t | r^4 = t^2 = (tr)^2 = \mathsf{id}\rangle$ be the dihedral group of order $8$ and $H$ be a subgroup of $G$. Let $\M_H(n)$ be the number of $H$-invariant domino tilings of $AD(n)$. These five symmetry classes are listed in Table \ref{tab.5sym}. We point out that, as mentioned in \cite[Section 7]{Ciucu97}, the first few terms of $\M_{H_4}(n)$ and $\M_{H_5}(n)$ do not all factor into small primes, so a simple product formula seems unlikely in these two unsolved symmetry classes.
\begin{table}[htb!]
\centering
\begin{tabular}{l|l|l}
  Subgroup of $G$ & Symmetry Class & Size and Reference \\
  \hline
  $H_1=\{\mathsf{id}\}$ & Original Aztec diamond & $\M_{H_1}(n)=2^{n(n+1)/2}$ (\cite{ELKP1,ELKP2}). \\
  \hline
  $H_2= \langle r \rangle$ & Quarter-turn invariant & $\M_{H_2}(n)$ is a product formula (\cite{Yang91}, \cite{Ciucu97}). \\
  \hline
  $H_3= \langle r^2 \rangle$ & Half-turn invariant & $\M_{H_3}(n)$ is a product formula (\cite{Yang91}, \cite{Ciucu97}).  \\
  \hline
  $H_4= \langle t \rangle$ & Diagonally symmetric & $\M_{H_4}(n)=2,6,24,132,1048,11960,190912,...$ \\
  \hline
  $H_5= \langle r^2,t \rangle$ & \makecell[l]{Diagonally and \\ anti-diagonally symmetric} & $\M_{H_5}(n)=2,4,10,28,96,384,1848,10432,...$ \\
\end{tabular}
\vspace{0.3cm}
\caption{The five symmetry classes of domino tilings of $AD(n)$.}\label{tab.5sym}
\end{table}

\subsection{Connection between ASMs and Aztec diamonds}\label{sec.ASMAD}

The connection between ASMs and Aztec diamonds was first mentioned in \cite{ELKP1,ELKP2}; the number of domino tilings of the Aztec diamond can be expressed as a weighted enumerations of ASMs depending on the number of $1$'s or $-1$'s in a matrix. This connection was made explicit by Ciucu \cite[Section 2]{Ciucu96}, the idea is stated below.

Given the Aztec diamond of order $n$, we consider the checkerboard coloring of the square lattice with the unit squares along its top right side colored black. Rotate the Aztec Diamond clockwise by $45$ degrees, black unit squares are now on the left and right sides. A \textit{cell} is a $2 \times 2$ square with left and right unit squares colored black while top and bottom colored white. We can think of $AD(n)$ as containing $n^2$ cells\footnote{This point of view is more clear if one works in the equivalent language of perfect matchings on the planar dual graph of $AD(n)$; see \cite{Ciucu96}.}. Figure \ref{fig.ADcell} shows an example when $n=4$, the four cells in the first row are displayed in red dotted edges.

Given a domino tiling of $AD(n)$, if we assign an entry $1, 0$ or $-1$ to each such cell where a cell contains $2, 1$ or $0$ complete domino(es), respectively, then we get a correspondence between a domino tiling of $AD(n)$ and an $n \times n$ ASM. Figures \ref{fig.ADtiling} and \ref{fig.ASM} illustrate an example of this correspondence. In Figure \ref{fig.ADtiling}, from left to right, three cells drawn in red dotted edges contain $1,0$ and $2$ complete domino(es), and therefore are assigned $0,-1$ and $1$, respectively. The corresponding ASM is given in Figure \ref{fig.ASM}.
\begin{figure}[htb!]
    \centering
    \subfigure[]
    {\label{fig.ADcell}\includegraphics[width=0.3\textwidth, trim=70 70 70 70]{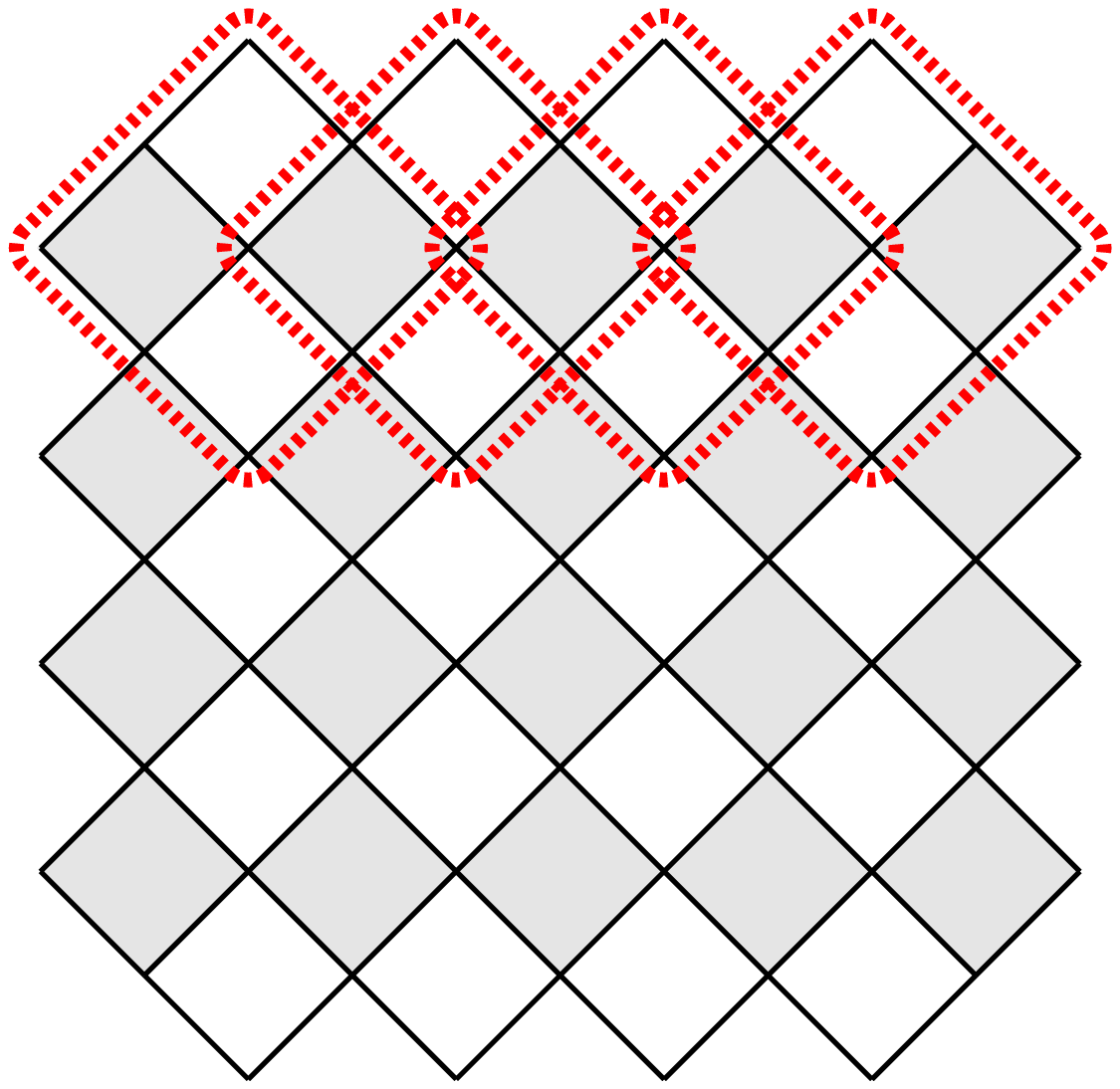}}
    \hspace{3mm}
    \subfigure[]
    {\label{fig.ADtiling}\includegraphics[width=0.3\textwidth, trim=70 70 70 70]{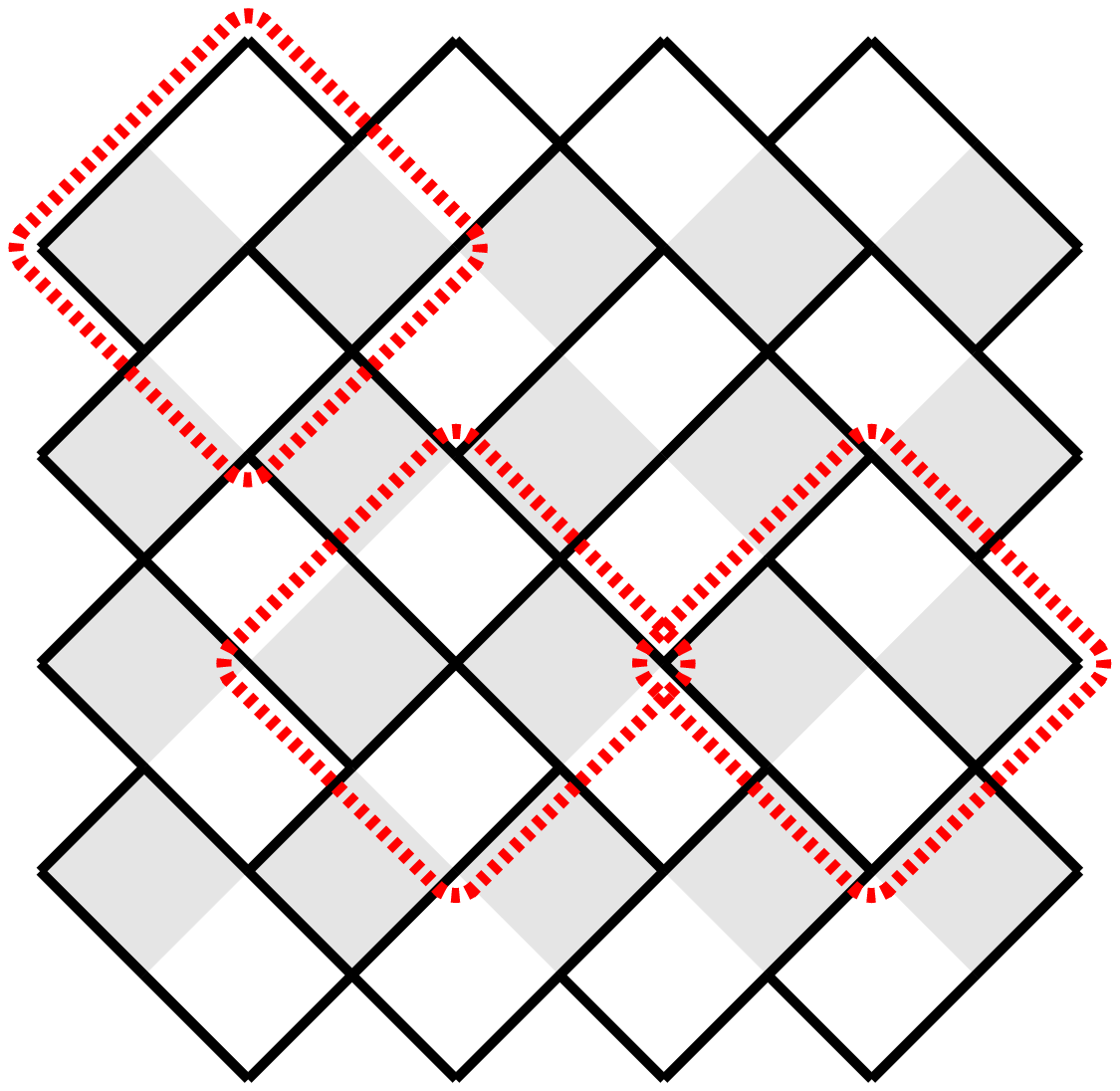}}
    \hspace{8mm}
    \subfigure[]
    {\label{fig.ASM}\raisebox{1.26cm}{\includegraphics[width=0.2\textwidth]{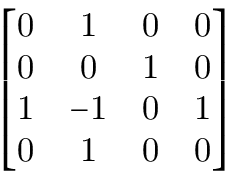}}}

    \caption{(a) The Aztec Diamond of order $4$ with top row cells drawn in red dotted edges. (b) A domino tiling of $AD(4)$. From left to right, the three cells drawn in red dotted edges are assigned $0,-1$ and $1$, respectively. (c) The corresponding ASM from the domino tiling in Figure \ref{fig.ADtiling}.}
\end{figure}

In \cite{Kup02}, Kuperberg introduced the \textit{off-diagonally symmetric} ASMs, which are the $2n \times 2n$ diagonally symmetric ASMs whose diagonal entries are all zeros. Motivated by this, using the correspondence discussed in the previous paragraph, we consider an analogue class of domino tilings of Aztec diamonds. For convenience, we do not perform a $45$-degree rotation on the Aztec diamond.
\begin{definition}\label{def.off}
  A domino tiling of $AD(n)$ is called \textit{off-diagonally symmetric} if
  \begin{itemize}
    \item the tiling is symmetric about the vertical diagonal of $AD(n)$, and
    \item the $n$ cells along the vertical diagonal are assigned $0$ in the above correspondence. In other words, each such cell contains exactly one complete domino.
  \end{itemize}
\end{definition}

An example of a diagonally symmetric domino tiling is shown in Figure \ref{fig.diagtiling}. From top to bottom, the six cells (drawn in red dotted edges) along the vertical diagonal are assigned $0,0,-1,1,0$ and $0$, respectively. On the other hand, Figure \ref{fig.offdiagtiling} gives an off-diagonally symmetric domino tiling; one can easily check that the six cells (drawn in red dotted edges) along the vertical diagonal all contain exactly one complete domino, and thus they are all assigned $0$.
\begin{figure}[htb!]
    \centering
    \subfigure[]
    {\label{fig.diagtiling}\includegraphics[width=0.4\textwidth]{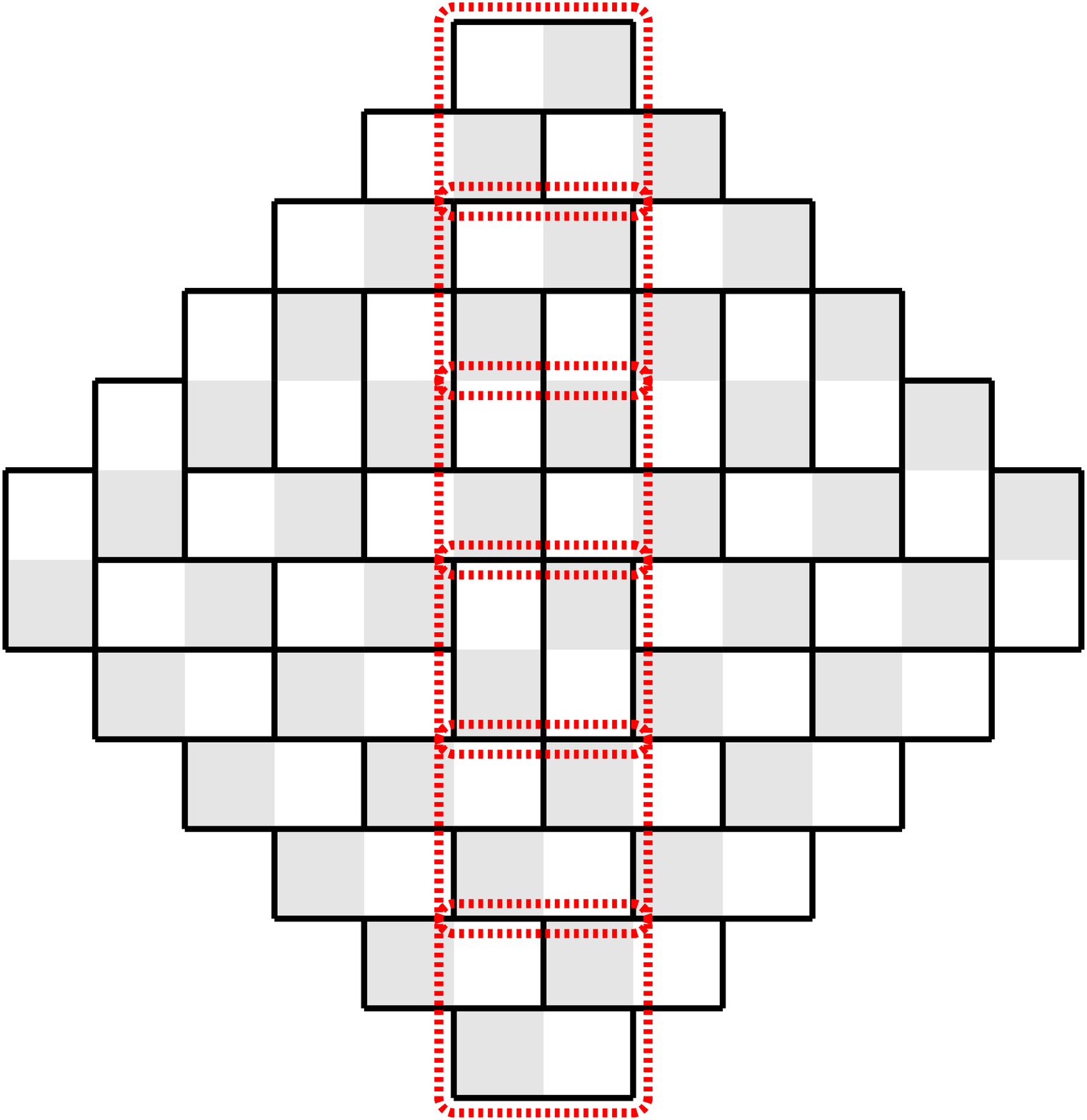}}
    \hspace{5mm}
    \subfigure[]
    {\label{fig.offdiagtiling}\includegraphics[width=0.4\textwidth]{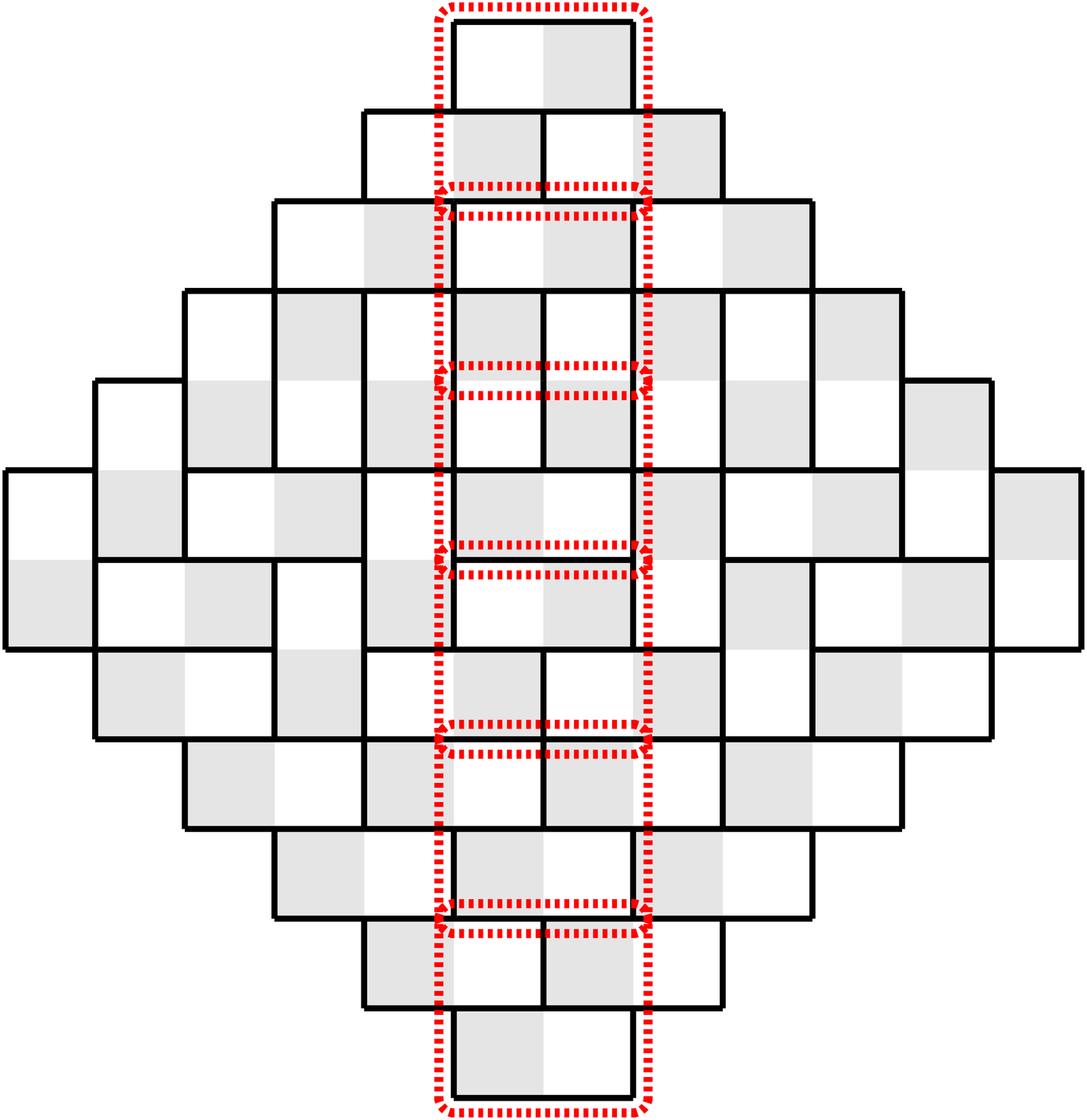}}
    \caption{(a) A diagonally symmetric domino tiling of $AD(6)$. (b) An off-diagonally symmetric domino tiling of $AD(6)$.}
    \label{fig.ADsymtiling}
\end{figure}

We also consider the case when there are some boundary defects. Label the unit squares on the southwestern boundary of $AD(n)$ by $1,2,\dotsc,n$ from bottom to top. By symmetry, if we remove one unit square from the southwestern boundary, then the corresponding unit square on the southeastern boundary also needs to be removed. Let $O(n;I)$ be the set of off-diagonally symmetric domino tilings of $AD(n)$ with all unit squares removed from the southwestern boundary \textbf{except} for those labeled $I = \{i_1,\dotsc,i_r\}$, where $1 \leq i_1 < \dotsc < i_r \leq n$. If there is no boundary defect (that is, $I = \{1,2,\dotsc n\}$), then we simply write $O(n;I)$ as $O(n)$.

In Figure \ref{fig.ADdent}, we provide an off-diagonally symmetric domino tiling of $AD(6)$ with unit squares labeled $3$ and $5$ removed. This is a tiling in the set $O(6;\{1,2,4,6\})$.
\begin{figure}[htb!]
    \centering
    \includegraphics[width=0.4\textwidth]{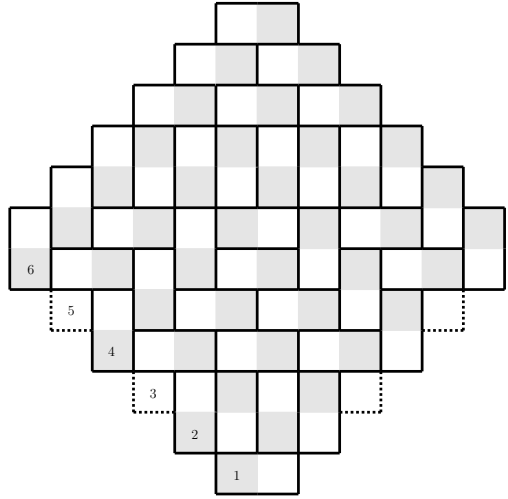}
    \caption{A tiling in the set $O(6;\{1,2,4,6\})$.}
    \label{fig.ADdent}
\end{figure}

\subsection{Main results}

Our first theorem (Theorem \ref{thm.main1}) provides a Pfaffian formula for enumerating off-diagonally symmetric domino tilings of the Aztec diamond with boundary defects.
\begin{theorem}\label{thm.main1}
  Let $I = \{i_1,\dotsc,i_r\}$, where $1 \leq i_1 < \dotsc < i_r \leq n$. Then there exists an infinite skew-symmetric matrix $A$ with integer entries such that
  \begin{equation}\label{eq.main1}
    |O(n;I)| = \pf(A_I),
  \end{equation}
  where $A_I$ is the matrix obtained from $A$ by selecting rows and columns indexed by $I$.
\end{theorem}
\begin{remark}
  We point out that this result follows from a modification of Stembridge's theorem \cite[Theorem 3.1]{Stem90}, in which he gives a Pfaffian expression for enumerating certain families of non-intersecting paths\footnote{We say two paths are \textit{non-intersecting} if they do not pass through the same vertex. A family of paths is non-intersecting if any two of the paths are non-intersecting.} (see Section \ref{sec.LGV} and Lemma \ref{lemma.stem}). However, each entry of the matrix $A$ obtained from a direct computation is slightly complicated, being given by a sum of determinants of $2 \times 2$ matrices whose entries are Delannoy numbers.
\end{remark}

The main contribution of this paper is given in the second theorem (Theorem \ref{thm.main2}). We find a nice way to express the entries of the matrix $A$, which leads to the simpler expression stated in Corollary \ref{cor.main}.

\begin{theorem}\label{thm.main2}
  The entries of the infinite skew-symmetric matrix\footnote{$A$ is skew-symmetric if $a_{i,j} = -a_{j,i}$ for $i,j \geq 1$.} $A = [a_{i,j}]_{i,j \geq 1}$ stated in Theorem \ref{thm.main1} are given by the following recurrence relation:
  \begin{equation}\label{eq.rec}
    \begin{cases}
    a_{1,j} = 2, & j > 1, \\
    a_{i,j} = a_{i-1,j} + a_{i,j-1} + a_{i-1,j-1}, & j>i+1,\quad i>1, \\
    a_{i,j} = a_{i-1,j} + a_{i-1,j-1} + 2(-1)^{i-1}, & j=i+1,\quad i>1, \\
    a_{i,j} = -a_{j,i}, & j \leq i.
    \end{cases}
  \end{equation}
\end{theorem}

It turns out that this implies the following result.
\begin{corollary}\label{cor.main}
  For $1 \leq i < j$, the entries $a_{i,j}$ of the matrix $A$ stated in Theorem \ref{thm.main2} can be expressed explicitly as an alternating sum of entries $s_{p,q}$ in the Schr{\"o}der triangle (see Section \ref{sec.sch}).
  \begin{equation}\label{eq.cor}
    a_{i,j} = 2 \sum_{\ell = 1}^{i}(-1)^{\ell-1}s_{i-\ell,j-\ell-1}.
  \end{equation}
\end{corollary}

Our data shows that the Pfaffians of the matrices $A_I$ do have large prime factors even for small values of the parameters, so having a nice product formula seems unlikely. However, the following striking pattern seems to hold: the number of off-diagonally symmetric domino tilings of $AD(2n)$ can be written as follows.
\begin{conjecture}\label{conj}
  There exists an integer sequence $\{o_n\}_{n \geq 0}$ so that
  \begin{equation}\label{eq.conj}
    |O(2n)| = 2^n o_{n-1}o_n.
  \end{equation}
  The first nine terms of $o_n$ are listed in Table \ref{tab.o_n}.
\end{conjecture}
This conjecture has been checked by computer up to $n=25$; see Section \ref{sec.open} for more discussions.
\begin{table}[htb!]
\centering
  \begin{tabular}{c|c|c|c|c|c|c|c|c|c}
    $n$ & $0$ & $1$ & $2$ & $3$ & $4$ & $5$ & $6$ & $7$ & $8$\\
    \hline
    $o_n$ & $1$ & $1$ & $3$ & $13$ & $149$ & $3 \times 887$ & $5 \times 29 \times 823$ & $29 \times 193 \times 1549$ & $3 \times 29 \times 263 \times 67049$ \\
  \end{tabular}
  \vspace{0.3cm}
  \caption{The first nine terms of $o_n$.}\label{tab.o_n}
\end{table}

The rest of this paper is organized as follows. In Section \ref{sec.pre}, we provide a brief introduction to Pfaffians, Delannoy numbers and the Schr{\"o}der triangle. We also review the main tools (the method of non-intersecting lattice paths, the Lindstr{\"o}m--Gessel--Viennot theorem and Stembridge's theorem) that will be used to prove our main results. In Section \ref{sec.enumerate}, we prove Theorem \ref{thm.main1} and discuss the diagonally symmetric domino tilings of $AD(n)$. The proof of Theorem \ref{thm.main2} and Corollary \ref{cor.main} will be given in Section \ref{sec.rec}. In Section \ref{sec.open}, we provide open problems related to the enumeration of off-diagonally symmetric domino tilings of $AD(n)$. In Appendix \ref{sec.eva}, we give a recursive way to calculate the Pfaffian of the matrix in Theorem \ref{thm.main2}.

\section{Preliminaries}\label{sec.pre}

We begin this section by introducing the notation of the Pfaffian (Section \ref{sec.pf}), Delannoy numbers and the Schr{\"o}der triangle (Section \ref{sec.sch}). In Section \ref{sec.path}, we review the method of non-intersecting lattice paths which turns the enumeration of tilings into the enumeration of non-intersecting lattice paths. The determinant and Pfaffian formulas for enumerating families of non-intersecting lattice paths will be stated in Section \ref{sec.LGV}.

\subsection{Pfaffians}\label{sec.pf}

Let $A = [a_{i,j}]_{i,j \geq 1}$ be an infinite matrix. If $I = \{i_1,\dotsc,i_r\}$ (resp., $J = \{j_1, \dotsc,j_r\}$) is a set of row (resp., column) indices, then we write $A_{I,J}$ for the $r \times r$ submatrix obtained from $A$ by choosing the rows indexed by $I$ and columns indexed by $J$. We say a matrix $A$ is \textit{skew-symmetric} if entries $a_{i,j}=-a_{j,i}$ for $i,j \geq 1$. In particular, the diagonal entries of $A$ are all zeros. For a skew-symmetric matrix $A$, we always take $I = J$. So, we simply write $A_{I}$ for $A_{I,I}$ if $A$ is skew-symmetric.

A \textit{perfect matching} or \textit{$1$-factor} $\sigma$ of $[2n] = \{1,2,\dotsc,2n\}$ is a partition of $[2n]$ into $2$-element blocks. A perfect matching can be written as $\sigma = \{(\sigma_1,\sigma_2),\dotsc,(\sigma_{2n-1},\sigma_{2n})\}$, where $\sigma_{2i-1}$ and $\sigma_{2i}$ ($\sigma_{2i-1}<\sigma_{2i}$) are in the same block for $i=1,2,\dotsc, n$, and $\sigma_{1} < \sigma_{3} < \sigma_{5} < \dotsc < \sigma_{2n-1}$. We write $\mathscr{F}_{2n}$ for the set of perfect matchings of $[2n]$.

Let $A$ be a $2n \times 2n$ skew-symmetric matrix, the Pfaffian of $A$ (see for instance \cite[Section~2]{Stem90} and \cite[Section~1]{ITZ13}) is defined to be
\begin{equation}\label{eq.pf}
  \pf(A)= \sum_{\sigma \in \mathscr{F}_{2n}} \sgn(\sigma) \prod_{i=1}^{n}a_{\sigma_{2i-1}, \sigma_{2i}},
\end{equation}
where the summation is over all perfect matchings $\sigma = \{(\sigma_1,\sigma_2),\dotsc,(\sigma_{2n-1},\sigma_{2n})\}$ of $[2n]$ described in the previous paragraph and $\sgn(\sigma)$ is given by the sign of the permutation written in the following two line notation
$$
\begin{pmatrix}
  1 & 2 & \cdots & 2n-1 & 2n \\
  \sigma_1 & \sigma_2 & \cdots & \sigma_{2n-1} & \sigma_{2n}
\end{pmatrix}.
$$

\subsection{Delannoy numbers and the Schr{\"o}der triangle}\label{sec.sch}

Consider an infinite triangular lattice $\mathcal{T}$, we choose a coordinate system on the triangular lattice (see Figure~\ref{fig.coordinate}) by fixing a lattice point as the origin, and letting the positive $x$-axis (resp., $y$-axis) be a lattice line pointing southeast (resp., northeast). Consequently, the edges that are parallel to the $x$-axis (resp., $y$-axis) are oriented southeast (resp., northeast), while the edges parallel to the line $y=x$ are oriented east.
\begin{figure}[htb!]
    \centering
    \subfigure[]
    {\label{fig.coordinate}\includegraphics[width=0.4\textwidth]{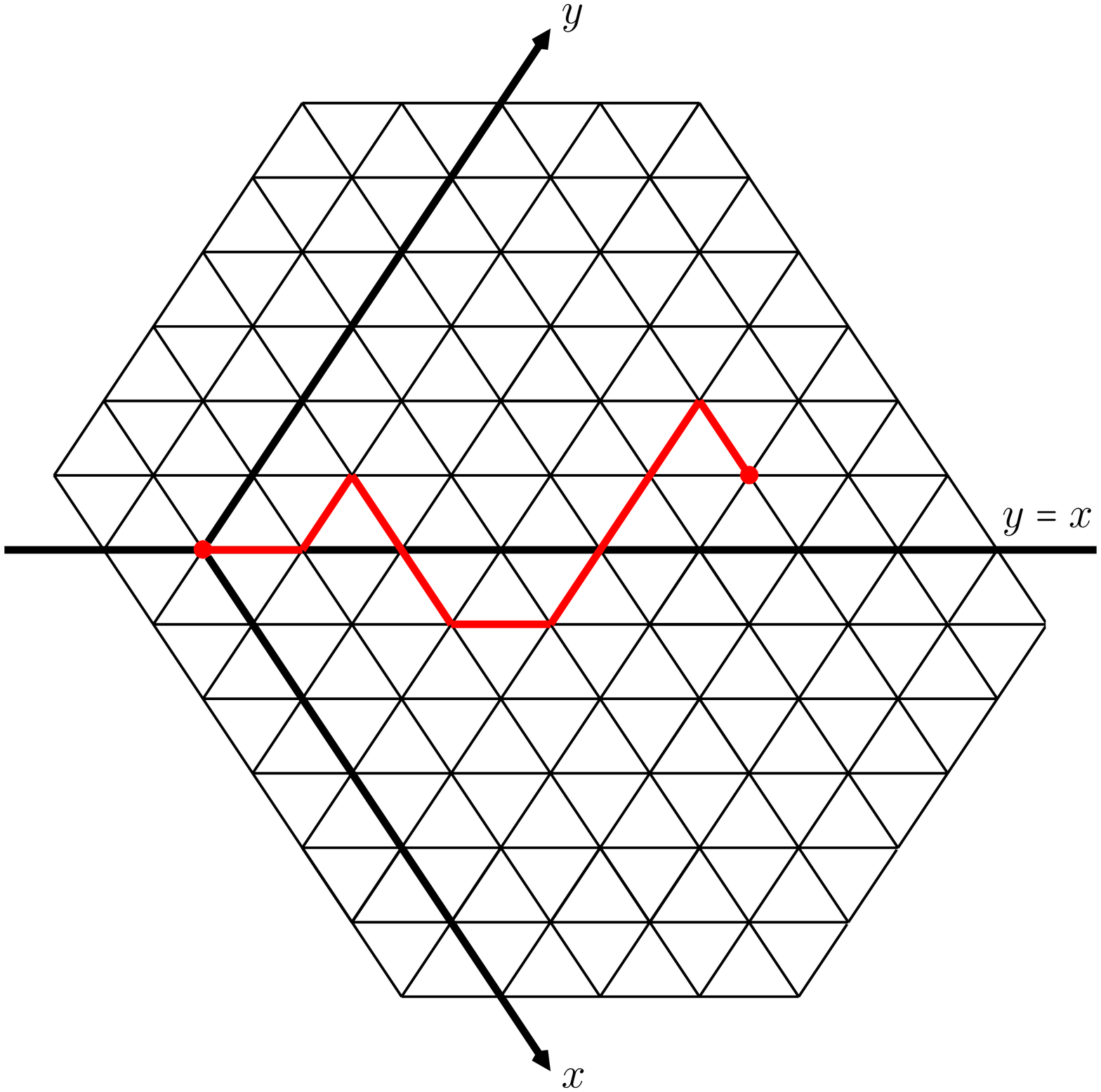}}
    \hspace{7mm}
    \subfigure[]
    {\label{fig.schroder}\raisebox{1.45cm}{
    \includegraphics[width=0.5\textwidth]{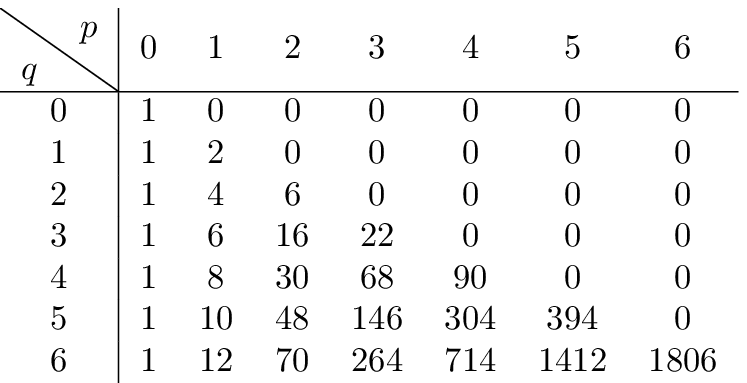}}}
    \caption{(a) The coordinate system on the triangular lattice. A Delannoy path (drawn in red) in the set $\mathscr{D}_{5,6}$. (b) The first few entries of the Schr{\"o}der triangle $\{s_{p,q}\}_{0 \leq p \leq q}$.}
\end{figure}

A \textit{Delannoy path} is a lattice path going from $(0,0)$ to $(p,q)$ ($p,q \geq 0$), using steps $(1,0)$, $(0,1)$ or $(1,1)$ on the triangular lattice $\mathcal{T}$. We write $\mathscr{D}_{p,q}$ for the set of Delannoy paths going from $(0,0)$ to $(p,q)$, see Figure \ref{fig.coordinate} for an example. The Delannoy number, denoted by $d_{p,q}$, is the size of $\mathscr{D}_{p,q}$. By convention, we set $d_{p,q} = 0$ if $p<0$ or $q<0$. The Delannoy numbers can be obtained recursively
\begin{equation}\label{eq.Drec}
  d_{p,q} = d_{p-1,q} + d_{p,q-1} + d_{p-1,q-1},
\end{equation}
with initial values $d_{p,0} = d_{0,q} = 1$ for $p,q \geq 0$. The closed-form expression (see for example \cite[Section 2]{BS05}) is given by
\begin{equation}\label{eq.Dclosed}
  d_{p,q} = \sum_{i=0}^{p}\binom{p}{i}\binom{q}{i}2^i.
\end{equation}

Now, we consider the set of lattice paths $\mathscr{S}_{p,q}$ from $(0,0)$ to $(p,q)$, $0 \leq p \leq q$, using steps $(1,0)$, $(0,1)$ or $(1,1)$ but never pass below the line $y=x$. If we write $s_{p,q}$ for the cardinality of $\mathscr{S}_{p,q}$, then $\{s_{p,q}\}_{0 \leq p \leq q}$ forms a triangular array (see Figure \ref{fig.schroder}) which is called the \textit{Schr{\"o}der triangle}; see for example \cite[A033877]{OEIS} and \cite[Section 2]{PS98}.

We note that $s_{p,q}$ can be expressed recursively as follows.
\begin{equation}\label{eq.recsch}
    \begin{cases}
    s_{0,q} = 1, & q \geq 0, \\
    s_{p,q} = s_{p-1,q} + s_{p,q-1} + s_{p-1,q-1}, & q > p > 0, \\
    s_{p,q} = s_{p-1,q} + s_{p-1,q-1}, & q = p > 0.
    \end{cases}
\end{equation}

\subsection{The method of non-intersecting lattice paths}\label{sec.path}

The method of non-intersecting lattice paths (or simply non-intersecting paths) (see~\cite[Section 3.1]{Propp15}) is one of the powerful techniques used to count domino or lozenge tilings. The core idea is to view such a tiling as a family of non-intersecting paths. We describe below how this method works for domino tilings of the Aztec diamond.

Given the Aztec diamond of order $n$ ($AD(n)$), we consider the checkerboard coloring mentioned in Section \ref{sec.ASMAD}. We mark the midpoint of the left edge of each black unit square, and join these midpoints by edges. Then we obtain a subgraph $\mathcal{AD}(n)$ of the triangular lattice $\mathcal{T}$ (see Figure~\ref{fig.ADtri} for $\mathcal{AD}(6)$).

Let $U = \{u_1,u_2,\dotsc,u_n\}$ be the collection of midpoints (from bottom to top) on the southwestern boundary and $W = \{w_1,w_2,\dotsc,w_n\}$ the collection of midpoints (from bottom to top) on the southeastern boundary. We label midpoints on the vertical diagonal of $AD(n)$ by $\{v_{2},v_4,\dotsc,v_{2n}\}$ from bottom to top, and the midpoints one unit left to the vertical diagonal by $\{v_{1},v_3,\dotsc,v_{2n-1}\}$ from bottom to top. Let $V = \{v_1,v_2,\dotsc,v_{2n}\}$ be the collection of these $2n$ midpoints (they are marked red in Figure~\ref{fig.ADtri}).
\begin{figure}[htb!]
    \centering
    \includegraphics[width=0.6\textwidth]{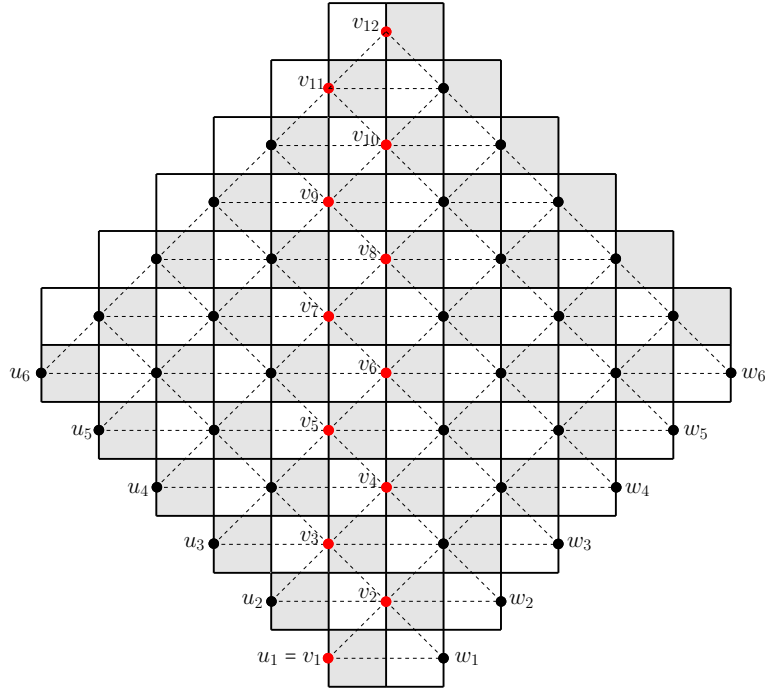}
    \caption{The subgraph $\mathcal{AD}(6)$ (dotted edges) of the triangular lattice on $AD(6)$.}
    \label{fig.ADtri}
\end{figure}

In general, there is a bijection between the set of domino tilings of a region $R$ on the square lattice and families of non-intersecting Delannoy paths with certain starting and ending points (determined by the region $R$). This bijection is implicit in the work of Sachs and Zernitz~\cite{SZ94}, and was made explicit by Randall (see \cite[Section 4]{Ciucu96} and \cite[Section 2]{LRS01}).

The bijection works as follows (see Figure~\ref{fig.ADpath} for two examples). Given a domino tiling, map each horizontal domino with the left unit square colored black to a $(1,1)$ step on the triangular lattice; map each vertical domino with the top (resp., bottom) unit square colored black to a $(1,0)$ (resp., $(0,1)$) step on the triangular lattice. Note that no step of the lattice paths corresponds to horizontal dominoes in which the right unit square is black.
\begin{figure}[htb!]
    \centering
    \subfigure[]
    {\label{fig.offdiagpath}\includegraphics[width=0.48\textwidth]{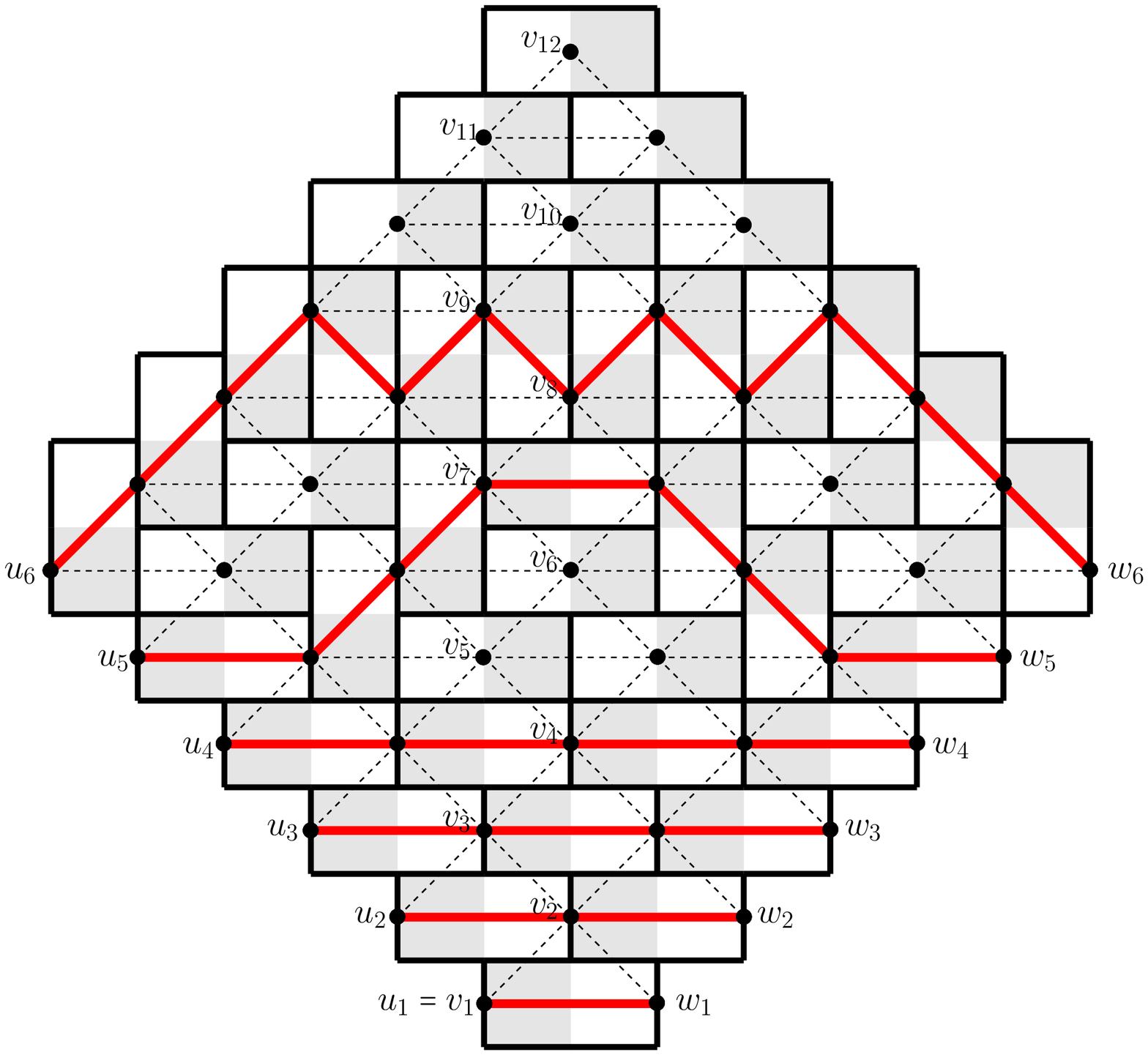}}
    \hspace{0mm}
    \subfigure[]
    {\label{fig.offdiagdentpath}\includegraphics[width=0.48\textwidth]{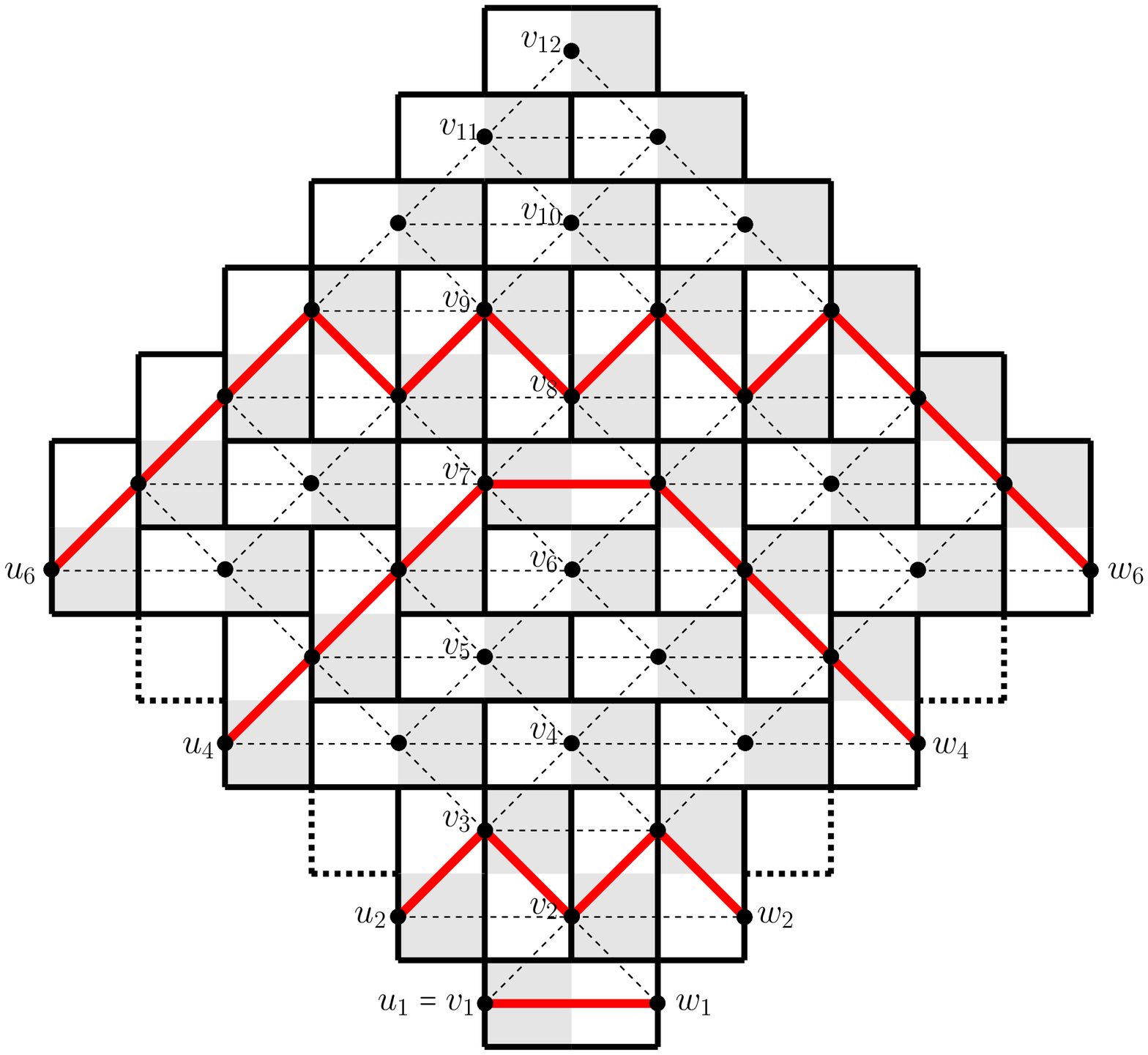}}
    \caption{(a) The non-intersecting Delannoy paths (in red) corresponding to the tiling given in Figure \ref{fig.offdiagtiling}. (b) The non-intersecting Delannoy paths (in red) corresponding to the tiling given in Figure \ref{fig.ADdent}.}
    \label{fig.ADpath}
\end{figure}

As a consequence, domino tilings of $AD(n)$ are in one-to-one correspondence with $n$-tuples of non-intersecting Delannoy paths in $\mathcal{AD}(n)$ connecting $U$ to $W$. Enumerating domino tilings of the Aztec diamond in this way was previously appeared in \cite{EF05} and later in \cite{BvL13}.

From the above discussion, it is easy to see that a diagonally symmetric domino tiling of $AD(n)$ corresponds to an $n$-tuple of non-intersecting Delannoy paths, all of which are symmetric about the vertical diagonal (see Figure~\ref{fig.ADpath}). Consequently, such a domino tiling corresponds to only the left (or right) half part of these $n$ non-intersecting Delannoy paths.

To be more precise, if the intersection of the vertical diagonal and a symmetric Delannoy path is the point $v_{2i}$, then we take $v_{2i}$ as the ending point of that path. For example, the paths going from $u_2,u_4$ and $u_6$ in Figure \ref{fig.offdiagpath} and the paths going from $u_2$ and $u_6$ in Figure \ref{fig.offdiagdentpath} belong to this case.

On the other hand, if the intersection of the vertical diagonal and a symmetric Delannoy path is not a lattice point of $\mathcal{AD}(n)$, then it must cross a $(1,1)$-step in that path. In this case, we take $v_{2i-1}$ as the ending point of that path. See the paths starting from $u_1,u_3$ and $u_5$ in Figure \ref{fig.offdiagpath} and the paths starting from $u_1$ and $u_4$ in Figure \ref{fig.offdiagdentpath} as examples.

Let $\mathcal{DS}(n)$ be the subgraph that consists of the ``left half'' part of $\mathcal{AD}(n)$; see Figure \ref{fig.tri-AD} for an example when $n=10$. Therefore, a diagonally symmetric domino tiling of the Aztec diamond of order $n$ corresponds to an $n$-tuple of non-intersecting Delannoy paths in $\mathcal{DS}(n)$ with the $n$ starting points being the points in $U$, and the $n$ ending points being selected from the $2n$ points in $V$.

Consider the $n$ cells along the vertical diagonal of $AD(n)$, the midpoints of the left edges of the two black unit squares in the $i$th cell are marked by $v_{2i-1}$ and $v_{2i}$ for $i=1,\dotsc,n$ (from bottom to top). We remind the reader that a cell is assigned $1, 0$ or $-1$ if it contains $2, 1$ or $0$ complete domino(es). The ``off-diagonal'' condition is characterized by the following lemma.
\begin{lemma}\label{lemma.diag0}
 The $i$th cell along the vertical diagonal of the Aztec diamond is assigned $0$ if and only if either $v_{2i-1},v_{2i} \in V_0$ or $v_{2i-1},v_{2i} \notin V_0$, where $ V_0 \subset V$ is the set of ending points of the corresponding family of Delannoy paths.
\end{lemma}
\begin{proof}
 We assume such a cell is assigned $0$, that is, it contains only one complete domino. Due to the symmetry, this domino is horizontal\footnote{If it were a vertical domino, then this cell must contain two complete vertical dominoes which contradicts to our assumption.}, which leads to the four possible cases shown in Figure \ref{fig.cell}. If the bottom (resp., top) half is covered by a horizontal domino, then there are two possible ways to cover the top left (resp., bottom left) corner in that cell; they are illustrated in the left (resp., right) two cases in Figure \ref{fig.cell}. According to the bijection between domino tilings of the Aztec diamond and families of non-intersecting Delannoy paths mentioned above, $v_{2i-1}$ and $v_{2i}$ are both (resp., neither) ending points of paths.

 Conversely, we first assume $v_{2i-1}$ and $v_{2i}$ are both ending points. According to the discussion about how we select the ending points, the fact that a path ends at $v_{2i-1}$ implies that the bottom half of this cell is covered by a horizontal domino. The path that ends at $v_{2i}$ produces a domino that is no completely contained in the cell (see left two cases in Figure \ref{fig.cell}).

 Next, we assume that neither of $v_{2i-1}$ and $v_{2i}$ are ending points. According to the bijection, no path ends at $v_{2i}$ implies that the top half of this cell is covered by a horizontal domino. Since no path ends at $v_{2i-1}$, there are to two possible tilings shown in the right two cases in Figure \ref{fig.cell}. In all these cases, the cell contains only one complete domino, and is therefore assigned $0$. This completes the proof.
\end{proof}
\begin{figure}[htb!]
    \centering
    \includegraphics[width=0.95\textwidth]{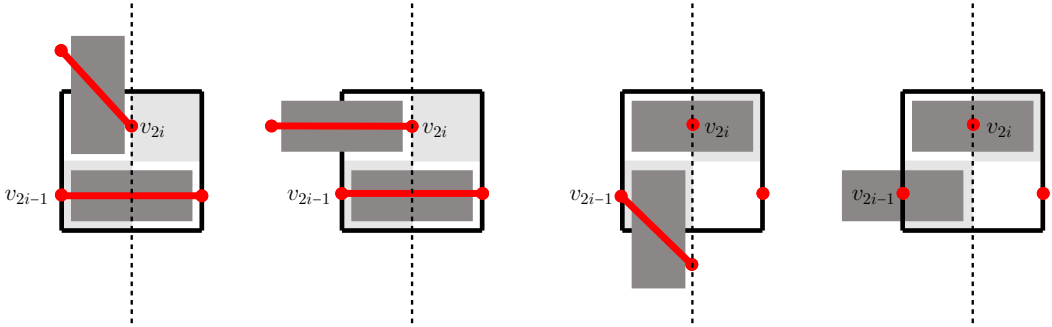}
    \caption{The four possible tilings of a cell which is assigned $0$ along the vertical diagonal and their corresponding paths.}
    \label{fig.cell}
\end{figure}

\subsection{Enumeration of families of non-intersecting lattice paths}\label{sec.LGV}

Throughout this paper, we write $\mathcal{G}$ for a locally finite, connected, directed acyclic graph, with a weight function $\wt$ that assigns elements in some commutative ring to each edge of $\mathcal{G}$. Let $U = \{u_1,\dotsc,u_n\}$ and $V = \{v_1,\dotsc,v_m\}, (m \geq n)$ be two sets of vertices of $\mathcal{G}$. The enumeration of families of non-intersecting paths connecting $U$ to $V$ has been studied by Lindstr{\"o}m \cite{Lin73} and later by Gessel and Viennot \cite{GV85} when $m=n$. Stembridge \cite{Stem90} generalized this result to the situation when the endpoints of paths are allowed to vary over a given set.

Two sets of vertices $U = \{u_1,\dotsc,u_n\}$ and $V = \{v_1,\dotsc,v_m\},(m \geq n)$ of $\mathcal{G}$ are said to be \textit{compatible} if, for any pair of paths $p_i$ from $u_i$ to $v_k$ and $p_j$ from $u_j$ to $v_{\ell}$ such that $i<j,k>\ell$, the paths $p_i$ and $p_j$ intersect. The key point for this condition is that when $U$ and $V$ are compatible, the $n$-tuples of non-intersecting paths only consist of paths connecting $u_i$ to $v_{j_i}$, where $1 \leq j_1 < j_2 < \dotsc < j_n \leq m$ for $i=1,\dotsc,n$.

We introduce the following notation.
\begin{itemize}
  \item $\mathscr{P}(u_i,v_j)$ denotes the set of paths going from $u_i \in U$ to $v_j \in V$.
  \item $\mathscr{P}(u_i,V)$ denotes the set of paths going from $u_i \in U$ to any $v \in V$.
  \item $\mathscr{P}(U,V)$ denotes the set of $n$-tuples of paths $(p_1,\dotsc,p_n)$, where $p_i \in \mathscr{P}(u_i,V)$ for $i=1,\dotsc,n$.
  \item $\mathscr{P}_0(U,V)$ is the subset of $\mathscr{P}(U,V)$ consisting of non-intersecting $n$-tuples of paths.
\end{itemize}
The weight of a path $p$ is defined to be $\wt(p) = \prod_{e}\wt(e)$, where the product is over all edges $e$ of the path $p$. For an $n$-tuple of paths $P=(p_1,\dotsc,p_n)$, the weight is given by $\wt(P) = \prod_{i=1}^{n} \wt(p_i)$. Given a set of ($n$-tuples of) paths $\mathscr{P}$, the total weight is given by  $GF(\mathscr{P}) = \sum_{P \in \mathscr{P}} \wt(P)$, the weighted sum of all the elements in $\mathscr{P}$.

The Lindstr{\"o}m--Gessel--Viennot theorem (Theorem \ref{thm.GV1}) gives the determinant formula for the number of families of non-intersecting paths $\mathscr{P}_0(U,V)$ when $|U|=|V|$. This is a powerful and elegant result with numerous applications in different contexts (see the survey paper written by Krattenthaler \cite[Section 10.13]{Kra17}).
\begin{theorem}[Lindstr{\"o}m~\cite{Lin73}; Gessel,Viennot~\cite{GV85}]\label{thm.GV1}
Consider a locally finite, connected, directed acyclic graph $\mathcal{G}$. Suppose $U = \{u_1,\dotsc,u_n\}$ and $V = \{v_1,\dotsc,v_m\}$ are two sets of vertices of $\mathcal{G}$ where $m=n$ and they are compatible. Then
  \begin{equation}\label{eq.GV2}
    GF(\mathscr{P}_{0}(U,V))  = \det \left( M \right),
  \end{equation}
    where the $(i,j)$-entry of the matrix $M$ is given by $GF(\mathscr{P}(u_i,v_j))$.
\end{theorem}

We consider the following setting for Stembridge's generalization. On the graph $\mathcal{G}$, we assume that $U = \{u_1,\dotsc,u_n\}$ is the set of fixed starting points of paths, and $V = \{v_1,\dotsc,v_m\}$ is the set of all possible ending points of paths, where $m \geq n$. For each pair $(i,j), 1 \leq i < j \leq n$, we write $Q_V(u_i,u_j)$ for the total weight of the set $\mathscr{P}_0(\{u_i,u_j\},V)$, that is, the set of pairs of non-intersecting paths $(p_i,p_j)$ where $p_i \in \mathscr{P}(u_i,V)$ and $p_j \in \mathscr{P}(u_j,V)$.

For $1 \leq i < j \leq n$, the formula for $Q_V(u_i,u_j)$ is obtained from summing over all possible pairs of non-intersecting paths whose ending points are in $V$. Each such possible cases is given by the Lindstr{\"o}m--Gessel--Viennot theorem (Theorem \ref{thm.GV1}), which leads to the following double summation.
\begin{equation}\label{eq.Q}
  Q_V(u_i,u_j) = \sum_{1 \leq k < \ell \leq m} \det
  \begin{bmatrix}
  GF\left( \mathscr{P}(u_i,v_{k}) \right) & GF\left( \mathscr{P}(u_i,v_{\ell}) \right) \\
  GF\left( \mathscr{P}(u_j,v_{k}) \right) & GF\left( \mathscr{P}(u_j,v_{\ell}) \right)
  \end{bmatrix}.
\end{equation}
We set $Q_V(u_i,u_j) = 0$ if $i=j$ by convention. When $i>j$, one can extend $Q_V(u_i,u_j)$ to be $-Q_V(u_j,u_i)$, this relation can be verified easily from \eqref{eq.Q}. Clearly, the matrix $[Q_V(u_i,u_j)]_{1 \leq i,j \leq n}$ is skew-symmetric. We state and sketch the proof of Stembridge's generalization below (see \cite[Theorem 3.1]{Stem90}).

\begin{theorem}[Stembridge~\cite{Stem90}]\label{thm.Stem}
   On a locally finite, connected, directed acyclic graph $\mathcal{G}$, let $U = \{u_1,\dotsc,u_n\}$ be the set of fixed starting points of paths and $V = \{v_1,\dotsc,v_m\}$ be the set of all possible ending points of paths, where $m \geq n$. Assume $U$ and $V$ are compatible and $n$ is even. Then we have
   \begin{equation}\label{eq.Stem}
     GF(\mathscr{P}_0(U,V)) = \pf [Q_V(u_i,u_j)]_{1 \leq i,j \leq n}.
   \end{equation}
\end{theorem}
\begin{remark}\label{rmk.stem}
  Stembridge remarked that when $n$ is odd, one may add a phantom vertex $u_{n+1}$ to the vertex set of $\mathcal{G}$ with no incident edges. Consider the set of starting points $U = \{u_1,\dotsc,u_{n+1}\}$. Define $Q_V(u_i,u_{n+1}) = GF(\mathscr{P}(u_i,V))$ for each $i=1,\dotsc,n$. Then Theorem \ref{thm.Stem} provides a Pfaffian of order $n+1$ for $GF(\mathscr{P}_0(U,V))$.
\end{remark}
\begin{proof} (Sketch) We give the key ideas of the proof here since the proof of Lemma \ref{lemma.stem} follows similar ideas, see \cite[Theorem 3.1]{Stem90} for details. By \eqref{eq.pf}, we can write
  \begin{equation}\label{eq.pfgen1}
    \pf [Q_{V}(u_i,u_j)]_{1 \leq i,j \leq n} = \sum_{\sigma \in \mathscr{F}_n} \sgn(\sigma) \prod_{(i,j) \in \sigma} Q_{V}(u_i,u_j),
  \end{equation}
where the summation is over all perfect matchings of $[n]$. This can be interpreted as the weighted sum of $(n+1)$-tuples $(\sigma,p_1,\dotsc,p_n)$ in which $\sigma \in \mathscr{F}_n$, and if $(i,j) \in \sigma$, then $p_i$ and $p_j$ must not intersect. The weight of $(\sigma,p_1,\dotsc,p_n)$ is naturally given by $\sgn(\sigma) \wt(p_1)\dotsc\wt(p_n)$.

We describe below a clever weight-preserving and sign-reversing involution $\phi$ on $\mathscr{F}_n \times \mathscr{P}(U,V)$, which is a key part of Stembridge's proof of Theorem \ref{thm.Stem}.

First, choose a total order of the vertices of the graph. Given a configuration $(\sigma,p_1,\dotsc,p_n)$ with at least one pair of intersecting paths. Find the least vertex $w$ which is a vertex of intersection of paths; if there are more than two paths meeting at $w$, then select the two distinct paths $p_i$ and $p_j$ with the smallest indices. This makes $w$, $p_i$ and $p_j$ unique for each configuration.

Second, create a new path $p_i^{\prime}$ (resp., $p_j^{\prime}$) by concatenating the first half of $p_i$ (resp., $p_j$) up to $w$ and the second half of $p_j$ (resp., $p_i$) after the vertex $w$, this has the effect of interchanging the ending points of $p_i$ and $p_j$. Set $p_k^{\prime} = p_k$ if $k \neq i,j$, then we obtain the new configuration $(\sigma^{\prime},p_1^{\prime},\dotsc,p_n^{\prime})$, where $\sigma^{\prime}$ is obtained by interchanging $i$ and $j$ in $\sigma$. The involution $\phi$ is defined by
\begin{equation}\label{eq.involution}
  \phi(\sigma,p_1,\dotsc,p_n) = (\sigma^{\prime},p_1^{\prime},\dotsc,p_n^{\prime}).
\end{equation}

Stembridge showed that the involution $\phi$ is well-defined, weight-preserving and sign-reversing. As a consequence, the terms involving intersecting configurations in \eqref{eq.pfgen1} are all canceled out, and the only contribution comes from non-intersecting configurations $\mathscr{P}_{0}(U,V)$. Therefore,
  \begin{equation}\label{eq.pfgen2}
    \pf [Q_{V}(u_i,u_j)]_{1 \leq i,j \leq n} = GF(\mathscr{P}_0(U,V)) \sum_{\sigma \in \mathscr{F}_n} \sgn(\sigma) = GF(\mathscr{P}_0(U,V)),
  \end{equation}
because it turns out that $\sum_{\sigma}\sgn(\sigma) = 1$ when $\sigma$ runs over all perfect matchings of $[n]$ (see \cite[Proposition 2.3 (c)]{Stem90}).
\end{proof}

\section{Enumeration of off-diagonally symmetric domino tilings}\label{sec.enumerate}

Stembridge's theorem (Theorem \ref{thm.Stem}) does not apply directly in our situation because of the special requirement on the ending points stated in Lemma \ref{lemma.diag0}. We provide below a modification of it in Section \ref{sec.modstem}. The proof of the first main result (Theorem \ref{thm.main1}) is given in Section \ref{sec.pfthm1}.

\subsection{Modification of Stembridge's theorem}\label{sec.modstem}

On a locally finite, connected, directed acyclic graph $\mathcal{G}$, we assume that $U = \{u_1,\dotsc,u_n\}$ is the set of fixed starting points of paths and $V = \{v_1,\dotsc,v_m\}$ is the set of all possible ending points of paths, where $m = 2n$. Let $v_{\ell}^{*} = \{v_{2\ell-1},v_{2\ell}\}$ for each $\ell=1,\dotsc,n$ and let $V^{*} = \{v_1^{*},v_{2}^{*},\dotsc,v_n^{*}\}$. We define the set $\mathscr{P}(U,V^{*})$ to be the subset of $\mathscr{P}(U,V)$ where the two elements in $v_{\ell}^{*}$ are both the ending points of paths or neither the ending points of paths, for all $\ell=1,\dotsc,n$. Similarly, $\mathscr{P}_0(U,V^{*})$ is the subset of $\mathscr{P}(U,V^{*})$ consisting of non-intersecting paths. Note that if $n$ is odd, then $\mathscr{P}(U,V^{*}) = \emptyset$.

We write $Q_{V^{*}}(u_i,u_j)$ for the total weight of the set $\mathscr{P}_0 (\{u_i,u_j\},V^{*})$, that is, the set of pairs of non-intersecting paths $(p_i,p_j)$ where $p_i$ starts at $u_i$, $p_j$ starts at $u_j$, and the ending points of $p_i$ and $p_j$ are both in $v_{\ell}^{*}$ for some $\ell=1,\dotsc,n$. The formula of $Q_{V^{*}}(u_i,u_j)$ can be obtained by an argument similar to the one that led to \eqref{eq.Q}:
\begin{equation}\label{eq.Q*}
  Q_{V^{*}}(u_i,u_j) = \sum_{1 \leq \ell \leq n} \det
  \begin{bmatrix}
  GF\left( \mathscr{P}(u_i,v_{2\ell -1}) \right) & GF\left( \mathscr{P}(u_i,v_{2\ell}) \right) \\
  GF\left( \mathscr{P}(u_j,v_{2\ell -1}) \right) & GF\left( \mathscr{P}(u_j,v_{2\ell}) \right)
  \end{bmatrix}.
\end{equation}

We have the following modification of Stembridge's theorem (Theorem \ref{thm.Stem}).
\begin{lemma}\label{lemma.stem}
  On a locally finite, connected, directed acyclic graph $\mathcal{G}$, let $U = \{u_1,\dotsc,u_n\}$ be the set of fixed starting points of paths and $V = \{v_1,\dotsc,v_{2n}\}$ be the set of all possible ending points of paths. Assume $U$ and $V$ are compatible and $n$ is even. Then we have
   \begin{equation}\label{eq.Stem}
     GF(\mathscr{P}_0(U,V^{*})) = \pf [Q_{V^{*}}(u_i,u_j)]_{1 \leq i,j \leq n}.
   \end{equation}
\end{lemma}
\begin{proof}
The proof of Lemma \ref{lemma.stem} is almost the same as the proof of Stembridge's theorem stated here as Theorem \ref{thm.Stem}; the only difference is that the set $V$ is replaced by $V^{*}$. Similarly, we can interpret $\pf [Q_{V^{*}}(u_i,u_j)]_{1 \leq i,j \leq n}$ as the weighted sum of $(n+1)$-tuples $(\sigma,p_1,\dotsc,p_n)$ in which $\sigma$ is a perfect matching of $[n]$, and if $(i,j) \in \sigma$, then $p_i$ and $p_j$ must not intersect; the additional condition is that the ending points of $p_i$ and $p_j$ must be in $v_{\ell}^{*}$ for some $\ell$.

Next, apply the involution $\phi$ mentioned in the proof of Theorem \ref{thm.Stem}; then we obtain a new configuration $(\sigma^{\prime},p_1^{\prime},\dotsc,p_n^{\prime})$, where $\sigma^{\prime}$ is obtained by interchanging $i$ and $j$ in $\sigma$. We only need to show that under this involution, if $(i_0, j_0) \in \sigma^{\prime}$, then the ending points of $p_{i_0}^{\prime}$ and $p_{j_0}^{\prime}$ lie in $v_{\ell}^{*}$ for some $\ell$.

The only cases in which this is not clear are those involving the modified paths $p_{i}^{\prime}$ and $p_{j}^{\prime}$. If $(k, j) \in \sigma^{\prime}$ for any $k \neq i,j$, then $(k,i) \in \sigma$. This implies that the ending points of $p_k$ and $p_i$ lie in $v_{\ell}^{*}$ for some $\ell$. After applying the involution, the ending point of the path $p_i$ becomes the ending point of the path $p_j^{\prime}$. Thus, the ending points of $p_k^{\prime}$ and $p_j^{\prime}$ are both in $v_{\ell}^{*}$ for some $\ell$, as desired. One may similarly argue that $(k, i) \in \sigma^{\prime}$ implies $p_k^{\prime}$ and $p_i^{\prime}$ are both in $v_{\ell}^{*}$ for some $\ell$. This shows the well-definedness of the involution on $\mathscr{P}(U,V^{*})$.

Following the same argument of the proof of Theorem \ref{thm.Stem}, we obtain the desired conclusion
  \begin{equation}\label{eq.pfgen3}
    \pf [Q_{V^{*}}(u_i,u_j)]_{1 \leq i,j \leq n} = GF(\mathscr{P}_0(U,V^{*})) \sum_{\sigma \in \mathscr{F}_n} \sgn(\sigma) = GF(\mathscr{P}_0(U,V^{*})).
  \end{equation}
This completes the proof of our modification of Stembridge's theorem.
\end{proof}

\subsection{Proof of Theorem \ref{thm.main1}}\label{sec.pfthm1}

We remind the reader that $O(n;I)$ is the set of off-diagonally symmetric domino tilings of $AD(n)$ with unit squares removed from the southwestern side except for those labeled $I = \{i_1,\dotsc,i_r\}$, where $1 \leq i_1 < \dotsc < i_r \leq n$. We continue with the same labeling of points on the graph $\mathcal{DS}(n)$ mentioned in Section \ref{sec.path} and keep the same notations of paths given in Sections \ref{sec.LGV} and \ref{sec.modstem}. It is not hard to see that $\mathcal{DS}(n)$ (edges are oriented in the same way as the triangular lattice $\mathcal{T}$) is directed acyclic and that two sets $U$ and $V$ are compatible.

\begin{proof}[Proof of Theorem \ref{thm.main1}]
  Suppose all the edge weights on the graph $\mathcal{DS}(n)$ are equal to $1$. We first assume $I = [n]$, that is, no unit square is removed from the southwestern boundary of $AD(n)$.

  From the discussion of Section \ref{sec.path} and Lemma \ref{lemma.diag0}, there is a bijection between the set $O(n;I)$ and the set of families of non-intersecting Delannoy paths in $\mathcal{DS}(n)$ whose starting points are in $U=\{u_{1},\dotsc,u_{n}\}$ and ending points are in pairs in some $v^{*}_{\ell} = \{v_{2\ell-1},v_{2\ell}\}$. That is, the set $\mathscr{P}_0(U,V^{*})$. Note that $\mathscr{P}_0(U,V^{*}) = \emptyset$ if $n$ is odd, so we assume $n$ is even.

  Let $\mathcal{G} = \mathcal{DS}(n)$ in Lemma \ref{lemma.stem}, then we obtain
  \begin{equation}\label{eq.pfmain1}
    |O(n;I)| = \left| \mathscr{P}_0(U,V^{*}) \right| = \pf [Q_{V^{*}}(u_i,u_j)]_{1 \leq i,j \leq n}.
  \end{equation}
  For convenience, let $A$ be the skew-symmetric matrix $[Q_{V^{*}}(u_i,u_j)]_{1 \leq i,j \leq n}$.

  Next, assume $I = \{i_1,\dotsc,i_r\}$, where $1 \leq i_1 < \dotsc < i_r \leq n$. In this case, $n-r$ unit squares are removed from the southwestern boundary. If the unit square labeled $\ell$ is removed, then $u_{\ell}$ is not the starting point of a path (see Figure \ref{fig.offdiagdentpath}). So, the set of starting points becomes $\{u_{i_1},\dotsc,u_{i_r} \}$. This leads to the bijection between the set $\mathscr{P}_0(\{u_{i_1},\dotsc,u_{i_r} \},V^{*})$ and the set $O(n;I)$.

  Similarly, by Lemma \ref{lemma.stem} with the set $U = \{u_{i_1},\dotsc,u_{i_r}\}$, then we have
  \begin{equation}\label{eq.pfmain2}
    |O(n;I)| = \left| \mathscr{P}_0(\{u_{i_1},\dotsc,u_{i_r} \},V^{*}) \right| = \pf (A_I),
  \end{equation}
  where $A_I$ is obtained from $A$ by selecting rows and columns indexed by $I = \{i_1,\dotsc,i_r\}$. This completes the proof of Theorem \ref{thm.main1}.
\end{proof}
\begin{remark}\label{rmk.diag}
We write $D(n;I)$ for the set of diagonally symmetric domino tilings of $AD(n)$ with unit squares removed from the southwestern side except for those labeled $I = \{i_1,\dotsc,i_r\}$, where $1 \leq i_1 < \dotsc < i_r \leq n$. There is a bijection between the set $D(n;I)$ and the set $\mathscr{P}_0(\{u_{i_1},\dotsc,u_{i_r} \},V)$ on the graph $\mathcal{DS}(n)$. So $|D(n;I)|$ has a similar Pfaffian expression which follows from the direct application of Stembridge's theorem (Theorem \ref{thm.Stem}). Then we have
\begin{equation}\label{eq.diag1}
  |D(n;I)| = \left| \mathscr{P}_0(\{u_{i_1},\dotsc,u_{i_r} \},V) \right| = \pf (B_I),
\end{equation}
where the matrix $B = [Q_V(u_i,u_j)]_{1 \leq i,j \leq n}$. This expression holds for all $n$ (even when $n$ is odd); see Remark \ref{rmk.stem}.

Our data shows that entries of the matrix $B$ are not so attractive (see \eqref{eq.matrixB} for $n=8$) and $\pf(B_I)$ does not factor into small primes for a given set $I$.
\begin{equation}\label{eq.matrixB}
B_{[8]}=
  \begin{bmatrix}
    0 & 6 & 18 & 46 & 114 & 278 & 674 & 1630 \\
    -6 & 0 & 18 & 70 & 202 & 526 & 1314 & 3222 \\
    -18 & -18 & 0 & 94 & 378 & 1134 & 3042 & 7742 \\
    -46 & -70 & -94 & 0 & 466 & 1966 & 6114 & 16830 \\
    -114 & -202 & -378 & -466 & 0 & 2438 & 10530 & 33502 \\
    -278 & -526 & -1134 & -1966 & -2438 & 0 & 12962 & 56982 \\
    -674 & -1314 & -3042 & -6114 & -10530 & -12962 & 0 & 69950 \\
    -1630 & -3222 & -7742 & -16830 & -33502 & -56982 & -69950 & 0
  \end{bmatrix}.
\end{equation}

However, the entries of the matrix $A$ look more interesting, we will analyze them and prove our main results in the next section.

\end{remark}

\section{Recursive entries of the matrix}\label{sec.rec}

Throughout this section, all the edge weights are equal to $1$. In order to analyze the entries $a_{i,j} = Q_{V^{*}}(u_i,u_j)$ of the matrix $A$, we begin with some auxiliary results on counting families of non-intersecting Delannoy paths (see Section \ref{sec.aux}). In Section \ref{sec.pfthm2}, we will use these lemmas to prove Theorem \ref{thm.main2} and Corollary \ref{cor.main}.

\subsection{Auxiliary results}\label{sec.aux}

The graph $\mathcal{DS}(n)$ that we previously considered has no edge connecting $u_{i}$ and $u_{i+1}$. We consider the graph $\overline{\mathcal{DS}}(n)$ obtained from $\mathcal{DS}(n)$ by adding edges connecting $u_{i}$ and $u_{i+1}$ for $i=1,\dotsc,n-1$; these edges are oriented southeast (see Figure \ref{fig.tri-inv} for $\overline{\mathcal{DS}}(10)$). It would be convenient to re-label the points on the southwestern boundary of $\overline{\mathcal{DS}}(n)$ by $x_1,x_2,\dotsc,x_n$, from bottom to top. The points that are adjacent to $x_i$ (except $x_{i-1}$ and $x_{i+1}$) are labeled by $y_{i-1}$ and $y_i$ for $i>1$ from bottom to top (see Figure \ref{fig.tri-par} for an illustration).

Notice that a Delannoy path starting at the point $x_i$ can end at a point from $v_1$ to $v_{2i}$. The number of Delannoy paths going from $x_i$ to $v_{\ell}$ is given by the Delannoy number (see \eqref{eq.Dclosed})
\begin{equation}\label{eq.aux1}
  |\mathscr{P}(x_i,v_{\ell})| = d_{\floor{\frac{2i-\ell}{2}},\floor{\frac{\ell}{2}}},
\end{equation}
where $1 \leq \ell \leq 2i$.
\begin{figure}[htb!]
    \centering
    \subfigure[]
    {\label{fig.tri-inv}\includegraphics[width=0.28\textwidth]{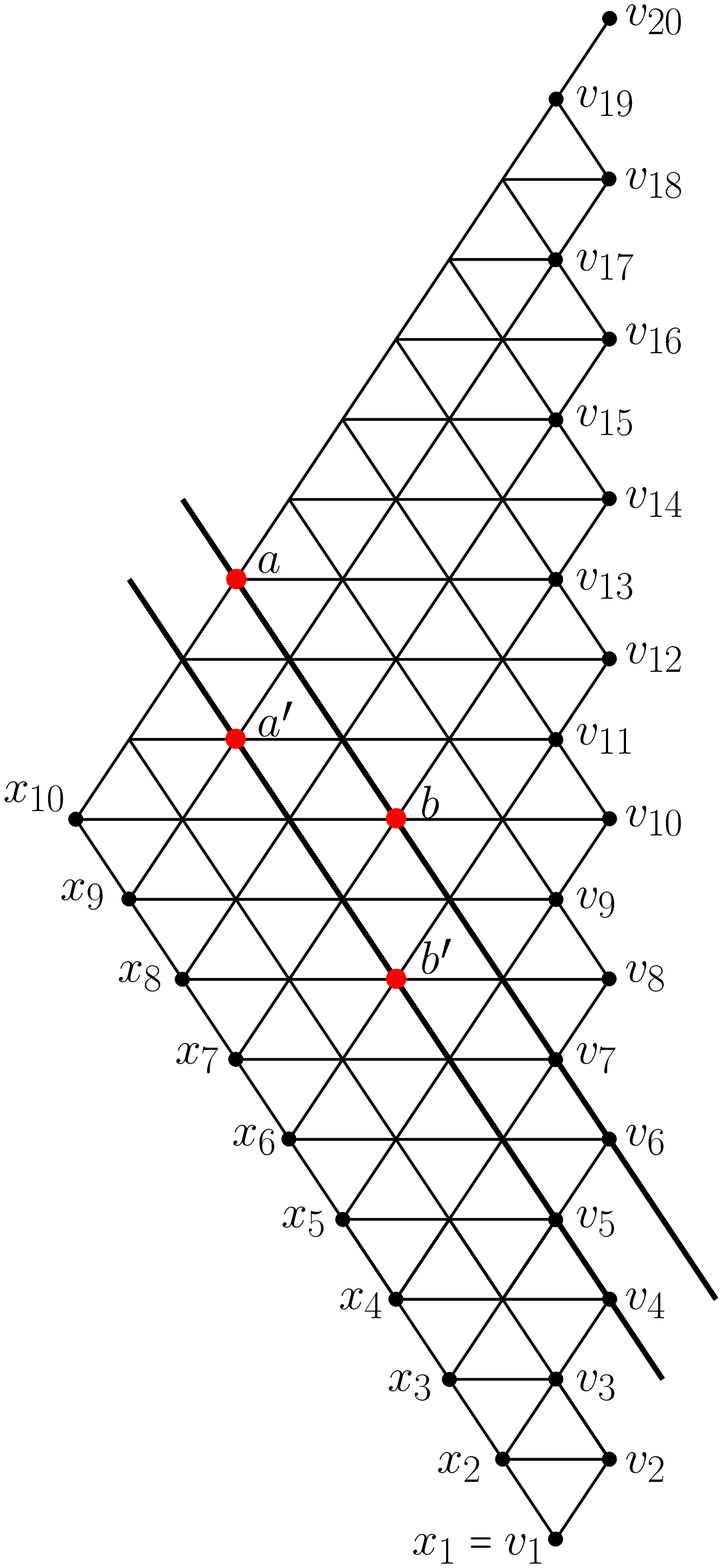}}
    \hspace{15mm}
    \subfigure[]
    {\label{fig.tri-par}\includegraphics[width=0.28\textwidth]{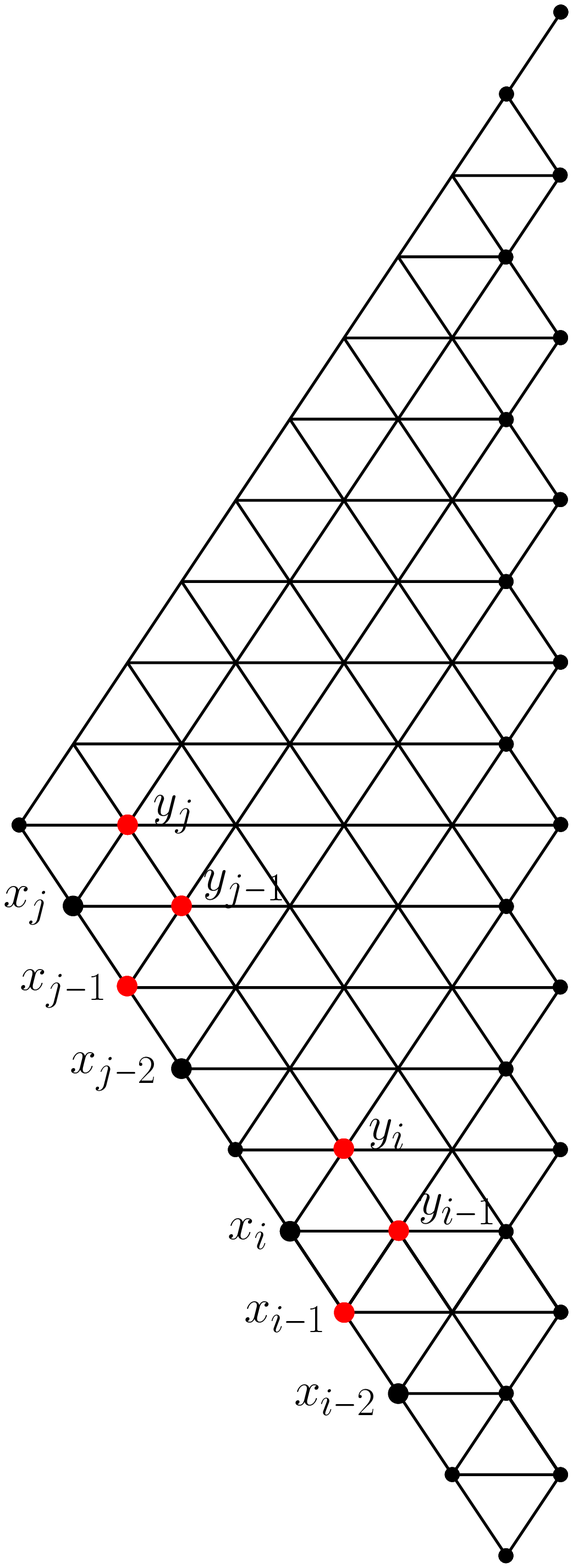}}
    \caption{(a) The graph $\overline{\mathcal{DS}}(10)$ with an illustration of points $a,b,a^{\prime},b^{\prime}$ in Lemma \ref{lemma.aux1}. (b) Partition of paths mentioned in Proposition \ref{prop.aux}.}
    \label{fig.tri1}
\end{figure}

Our first lemma shows that the number of pairs non-intersecting paths in $\overline{\mathcal{DS}}(n)$ with two starting points on the same lattice line $y = \ell$ is invariant under a certain translation of these two starting points. It is stated below.
\begin{lemma}\label{lemma.aux1}
  On the graph $\overline{\mathcal{DS}}(n)$, let $a = (p,\ell)$ and $b = (q,\ell)$ be two distinct points on the lattice line $y = \ell$. Let $a^{\prime} = (p+1,\ell-1)$ and $b^{\prime} = (q+1,\ell-1)$ be the points obtained from $a$ and $b$ by shifting downward one lattice line (we assume that $a^{\prime}$ and $b^{\prime}$ are still contained in $\overline{\mathcal{DS}}(n)$; see Figure \ref{fig.tri-inv}). Then
  \begin{equation}\label{eq.aux2}
    Q_{V^{*}}(a,b) = Q_{V^{*}}(a^{\prime},b^{\prime}).
  \end{equation}
\end{lemma}
\begin{proof}
Without loss of generality assume $p<q$. We observe that the number of paths going from $a$ to $v_{i+2}$ is the same as the number of paths going from $a^{\prime}$ to $v_{i}$ for all $i$. In other words, $|\mathscr{P}(a,v_{i+2})| = |\mathscr{P}(a^{\prime},v_{i})|$ for all $i$. Similarly, $|\mathscr{P}(b,v_{i+2})| = |\mathscr{P}(b^{\prime},v_{i})|$ for all $i$.

We list the numbers of paths going from these four points $a,b,a^{\prime}$ and $b^{\prime}$ to any $v \in V$ in Table \ref{tab.auxgv1}, where $a_{k}$ is non-zero for $k=1,\dotsc,\ell$, and $b_{k}$ is non-zero for $k=1,\dotsc,k_0$ for some $k_0 < \ell$ (depending on the position of $b$). The numbers of paths going from $a^{\prime}$ and $b^{\prime}$ to any $v \in V$ are obtained from the numbers of paths going from $a$ and $b$ to any $v \in V$ by shifting two entries to their left in Table \ref{tab.auxgv1}.
\begin{table}[]
\centering
\begin{tabular}{c|c|cc|cc|cc|c|cc|cc|cc|c}
    & $\dots$ & $v_{2i-3}$ & $v_{2i-2}$ & $v_{2i-1}$ & $v_{2i}$ & $v_{2i+1}$ & $v_{2i+2}$ & \dots & $v_{2j-3}$ & $v_{2j-2}$ & $v_{2j-1}$ & $v_{2j}$ & $v_{2j+1}$ & $v_{2j+2}$ & $\dots$\\
  \hline
  $a$ & $\dots$ & 0 & 0 & 0 & $a_1$ & $a_2$ & $a_3$ & $\dots$ & $a_{\ell-3}$ & $a_{\ell-2}$ & $a_{\ell-1}$ & $a_{\ell}$ & 0 & 0 & $\dots$ \\
  $b$ & $\dots$ & 0 & 0 & 0 & $b_1$ & $b_2$ & $b_3$ & $\dots$ & $b_{\ell-3}$ & $b_{\ell-2}$ & $b_{\ell-1}$ & $b_{\ell}$ & 0 & 0 & $\dots$ \\
  \hline
  $a^{\prime}$ & $\dots$ & 0 & $a_1$ & $a_2$ & $a_3$ & $a_4$ & $a_5$ & $\dots$ & $a_{\ell-1}$ & $a_{\ell}$ & 0 & 0 & 0 & 0 & $\dots$ \\
  $b^{\prime}$ & $\dots$ & 0 & $b_1$ & $b_2$ & $b_3$ & $b_4$ & $b_5$ & $\dots$ & $b_{\ell-1}$ & $b_{\ell}$ & 0 & 0 & 0 & 0 & $\dots$ \\
\end{tabular}
\vspace{0.3cm}
\caption{We assume $i<j$. In the first two rows, the numbers of paths going from $a$ and $b$ to some $v \in V$ in the graph $\overline{\mathcal{DS}}(n)$ are illustrated by $a_k$ and $b_k$ ($k=1,2,\dotsc,\ell$), respectively. The last two rows show the numbers of paths going from $a^{\prime}$ and $b^{\prime}$ to some $v \in V$ in the same graph. We display the columns in this way so that each $2 \times 2$ array is the matrix shown in the summand of \eqref{eq.Q*}.}\label{tab.auxgv1}
\end{table}

It follows from \eqref{eq.Q*} that $Q_{V^{*}}(a,b)$ is the summation of determinants of $2 \times 2$ matrices which are illustrated as $2 \times 2$ arrays in Table \ref{tab.auxgv1}. The non-zero $2 \times 2$ arrays coming from $a,b$ and $a^{\prime},b^{\prime}$ are identical. Summing over all determinants of these $2 \times 2$ arrays, we obtain the desired identity $Q_{V^{*}}(a,b) = Q_{V^{*}}(a^{\prime},b^{\prime})$.
\end{proof}

Next, we provide the result of $Q_{V^{*}}(x_i,x_j)$ when $i=1$ and $i=2$ in the following lemma.
\begin{lemma}\label{lemma.aux2}
  On the graph $\overline{\mathcal{DS}}(n)$, we have
  \begin{enumerate}
    \item $Q_{V^{*}}(x_1,x_j) = 2(j-1)$ for $j>1$.
    \item $Q_{V^{*}}(x_2,x_j) = 2(j-2)(j-1)$ for $j>2$.
    \item $Q_{V^{*}}(x_2,x_j) = Q_{V^{*}}(x_2,x_{j-1}) + Q_{V^{*}}(x_1,x_{j-1}) + Q_{V^{*}}(x_1,x_{j})-2$ for $j>2$.
  \end{enumerate}
\end{lemma}
\begin{proof}
The numbers of paths going from $x_1,x_2$ and $x_j$ to any $v \in V$ (obtained from \eqref{eq.aux1}) is listed in Table \ref{tab.auxgv2}, where $d_{p,q}$ denotes the Delannoy number, explicit expression is given in \eqref{eq.Dclosed}.
\begin{table}[H]
\centering
\begin{tabular}{c|cc|cc|cc|c}
        & $v_1$ & $v_2$ & $v_3$ & $v_4$ & $v_5$ & $v_6$ & $\cdots$  \\
  \hline
  $x_1$ & 1 & 1 & 0 & 0 & 0 & 0 & $\cdots$  \\
  \hline
  $x_2$ & 1 & 3 & 1 & 1 & 0 & 0 & $\cdots$  \\
  \hline
  $x_j$ & $d_{j-1,0}$ & $d_{j-1,1}$ & $d_{j-2,1}$ & $d_{j-2,2}$ & $d_{j-3,2}$ & $d_{j-3,3}$ & $\cdots$
\end{tabular}
\vspace{0.3cm}
\caption{The numbers of Delannoy paths going from $x_1,x_2$ and $x_j$ to any $v \in V$ in the graph $\overline{\mathcal{DS}}(n)$.}\label{tab.auxgv2}
\end{table}

The first two parts follow from the direct computation using formula \eqref{eq.Q*}. If $j>1$, then
\begin{equation*}
  Q_{V^{*}}(x_1,x_j) = \det \begin{bmatrix}
                              1 & 1 \\
                              d_{j-1,0} & d_{j-1,1}
                            \end{bmatrix}
                     = 2(j-1).
\end{equation*}
If $j>2$, then we have
\begin{equation*}
  Q_{V^{*}}(x_2,x_j) = \det \begin{bmatrix}
                              1 & 3 \\
                              d_{j-1,0} & d_{j-1,1}
                            \end{bmatrix}
                     + \det \begin{bmatrix}
                              1 & 1 \\
                              d_{j-2,1} & d_{j-2,2}
                            \end{bmatrix}
                     =2(j-2)(j-1).
\end{equation*}

The last part follows immediately by verifying the following identity for $j>2$.
\begin{equation*}
  2(j-2)(j-1) = 2(j-3)(j-2) + 2(j-2) + 2(j-1) - 2.
\end{equation*}

\end{proof}

The next result plays a major role in proving Theorem \ref{thm.main2}.
\begin{proposition}\label{prop.aux}
    On the graph $\overline{\mathcal{DS}}(n)$, we have the following recursive expression for $1<i<j$.
    \begin{equation}\label{eq.auxrec}
      Q_{V^{*}}(x_i,x_j) = Q_{V^{*}}(x_{i-1},x_{j-1}) + Q_{V^{*}}(x_{i},x_{j-1}) + Q_{V^{*}}(x_{i-1},x_{j})+2(-1)^{i-1}.
    \end{equation}
    In particular, if $j=i+1$, then $Q_{V^{*}}(x_{i},x_{j-1})=0$ which leads to
    \begin{equation}\label{eq.auxrec0}
      Q_{V^{*}}(x_i,x_{i+1}) = Q_{V^{*}}(x_{i-1},x_{i}) + Q_{V^{*}}(x_{i-1},x_{i+1})+2(-1)^{i-1}.
    \end{equation}
\end{proposition}
\begin{proof}
One can partition the set of paths starting from $x_i$ into the disjoint union of three sets, based on three possible first steps of a path; see Figure \ref{fig.tri-par}. In other words, when $i>1$, we have
\begin{equation*}
  \mathscr{P}(x_i,v) = \mathscr{P}(y_i,v) \cup \mathscr{P}(y_{i-1},v) \cup \mathscr{P}(x_{i-1},v),
\end{equation*}
and therefore
\begin{equation}\label{eq.partition}
  |\mathscr{P}(x_i,v)| = |\mathscr{P}(y_i,v)| + |\mathscr{P}(y_{i-1},v)| + |\mathscr{P}(x_{i-1},v)|,
\end{equation}
for any $v \in V$.

In Figure \ref{fig.tri-par}, the points $y_i$'s are on the same lattice line (say, $y=\ell$) while the points $x_i$'s are located one lattice line below ($y=\ell-1$). Assume $2<i<j$, we have the following four identities by Lemma \ref{lemma.aux1}.
\begin{align}
  Q_{V^{*}}(y_{i},y_{j}) & = Q_{V^{*}}(x_{i-1},x_{j-1}). \label{eq.Qid1}\\
  Q_{V^{*}}(y_{i-1},y_{j}) & = Q_{V^{*}}(x_{i-2},x_{j-1}). \label{eq.Qid2}\\
  Q_{V^{*}}(y_{i},y_{j-1}) & = Q_{V^{*}}(x_{i-1},x_{j-2}). \label{eq.Qid3}\\
  Q_{V^{*}}(y_{i-1},y_{j-1}) & = Q_{V^{*}}(x_{i-2},x_{j-2}). \label{eq.Qid4}
\end{align}
Using \eqref{eq.Q*}, \eqref{eq.partition} and the linearity of determinants, we obtain
\begin{align}
  Q_{V^{*}}(x_i,x_{j-1}) & = Q_{V^{*}}(y_i,x_{j-1}) + Q_{V^{*}}(y_{i-1},x_{j-1}) + Q_{V^{*}}(x_{i-1},x_{j-1}), \label{eq.Qpar1}\\
  Q_{V^{*}}(x_{i-1},x_{j}) & = Q_{V^{*}}(x_{i-1},y_{j}) + Q_{V^{*}}(x_{i-1},y_{j-1}) + Q_{V^{*}}(x_{i-1},x_{j-1}). \label{eq.Qpar2}
\end{align}
For a pair of paths starting from $x_i$ and $x_j$, each path can be partitioned into three disjoint sets mentioned above, this gives nine disjoint pairs of paths. So, we can rewrite $Q_{V^{*}}(x_i,x_j)$ into the sum of following nine terms.
\begin{align}\label{eq.Qnine}
  Q_{V^{*}}(x_i,x_j) & = Q_{V^{*}}(y_i,y_j) + Q_{V^{*}}(y_i,y_{j-1}) + Q_{V^{*}}(y_i,x_{j-1}) \nonumber\\
                     & + Q_{V^{*}}(y_{i-1},y_j) + Q_{V^{*}}(y_{i-1},y_{j-1}) + Q_{V^{*}}(y_{i-1},x_{j-1}) \nonumber\\
                     & + Q_{V^{*}}(x_{i-1},y_j) + Q_{V^{*}}(x_{i-1},y_{j-1}) + Q_{V^{*}}(x_{i-1},x_{j-1}).
\end{align}

If we view the nine terms displayed on the right hand side of \eqref{eq.Qnine} as a $3 \times 3$ array, then we can combine three terms in the last column by \eqref{eq.Qpar1} and combine three terms in the last row by \eqref{eq.Qpar2}; note that $Q_{V^{*}}(x_{i-1},x_{j-1})$ is counted twice. The four terms on the top left $2 \times 2$ array can be rewritten by identities \eqref{eq.Qid1},\eqref{eq.Qid2},\eqref{eq.Qid3} and \eqref{eq.Qid4}. So, we can simplify the right hand side of \eqref{eq.Qnine} to
\begin{align}
  Q_{V^{*}}(x_i,x_j) & = Q_{V^{*}}(x_i,x_{j-1}) + Q_{V^{*}}(x_{i-1},x_{j}) - Q_{V^{*}}(x_{i-1},x_{j-1}) \nonumber \\
  & + Q_{V^{*}}(x_{i-1},x_{j-1}) + Q_{V^{*}}(x_{i-1},x_{j-2}) + Q_{V^{*}}(x_{i-2},x_{j-1}) + Q_{V^{*}}(x_{i-2},x_{j-2}). \label{eq.Qsimplify}
\end{align}

Compare the equation \eqref{eq.auxrec} that we want to prove with \eqref{eq.Qsimplify}, it is equivalent to show that
\begin{equation}\label{eq.Qcompare}
  Q_{V^{*}}(x_{i-1},x_{j-1}) = Q_{V^{*}}(x_{i-2},x_{j-2})+Q_{V^{*}}(x_{i-1},x_{j-2})+Q_{V^{*}}(x_{i-2},x_{j-1}) + 2(-1)^{i-2}.
\end{equation}
We then continue with the same process of partitioning the paths starting from $x_{i-1}$ and $x_{j-1}$, and working through the simplification mentioned above. Finally, the problem is reduced to the case when $i=2$ and $j>i$:
\begin{equation*}
  Q_{V^{*}}(x_{2},x_{j}) = Q_{V^{*}}(x_{1},x_{j-1})+Q_{V^{*}}(x_{2},x_{j-1})+Q_{V^{*}}(x_{1},x_{j}) + 2(-1)^{2-1},
\end{equation*}
which has been shown in Lemma \ref{lemma.aux2}. This completes the proof of Proposition \ref{prop.aux}.
\end{proof}

\subsection{Proof of Theorem \ref{thm.main2} and Corollary \ref{cor.main}}\label{sec.pfthm2}

Now, we consider the graph $\mathcal{DS}(n)$ mentioned in Section \ref{sec.path}, see again Figure \ref{fig.tri-AD} for $n=10$. If we delete points $u_i$'s and $v_2$ and all their incident edges from $\mathcal{DS}(n)$, then we obtain the graph $\overline{\mathcal{DS}}(n-1)$. On the graph $\mathcal{DS}(n)$, we label the two points which are adjacent to $u_i$ (except $u_1$ and $u_2$) by $x_{i-2}$ and $x_{i-1}$, from bottom to top, respectively.

One can partition the set of paths starting from $u_i$ into the disjoint union of two sets, based on two possible first steps of a path; see Figure \ref{fig.tri-AD-par}. In other words, when $i>2$, we have
\begin{equation*}
  \mathscr{P}(u_i,v) = \mathscr{P}(x_{i-1},v) \cup \mathscr{P}(x_{i-2},v),
\end{equation*}
and therefore
\begin{equation}\label{eq.partition2}
  |\mathscr{P}(u_i,v)| = |\mathscr{P}(x_{i-1},v)| + |\mathscr{P}(x_{i-2},v)|,
\end{equation}
for any $v \in V$.
\begin{figure}[htb!]
    \centering
    \subfigure[]
    {\label{fig.tri-AD}\includegraphics[width=0.3\textwidth]{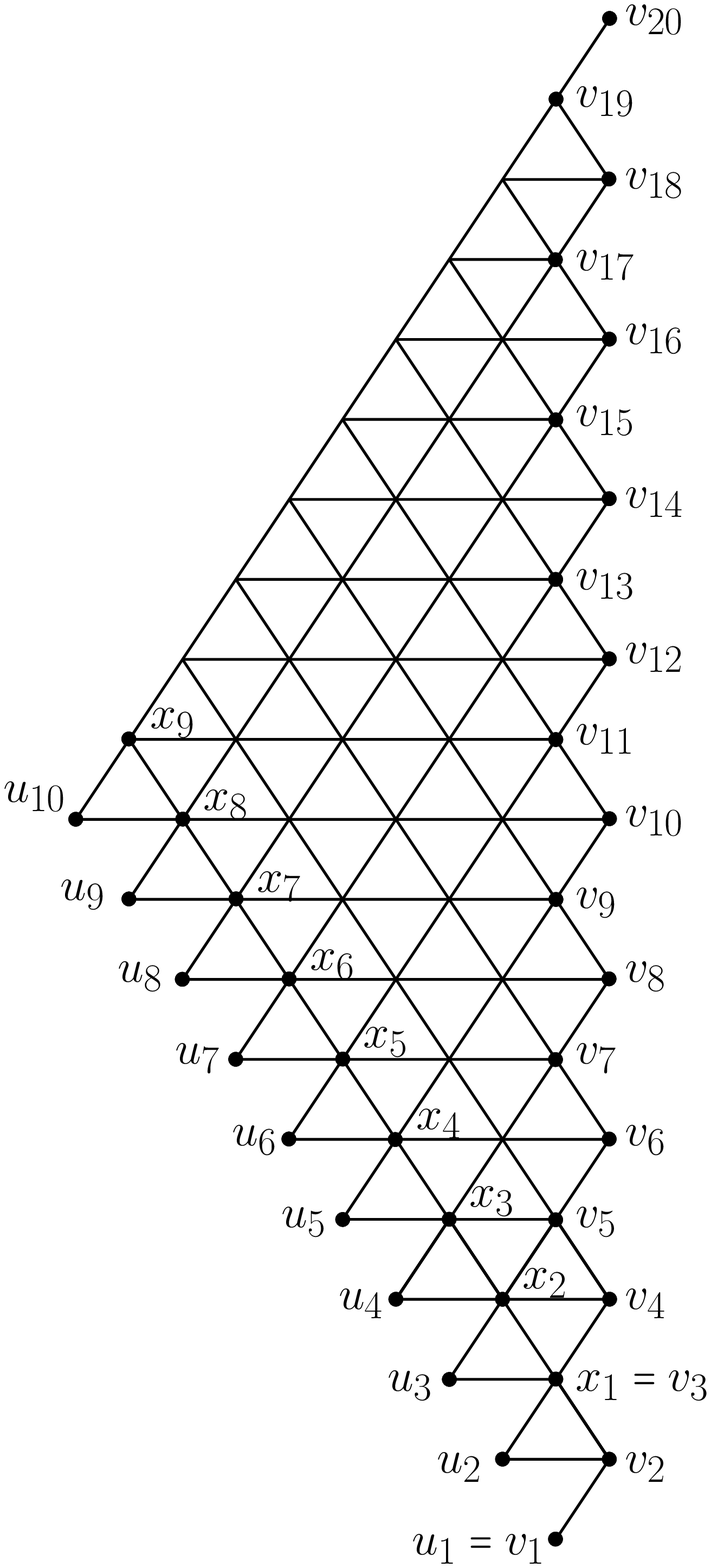}}
    \hspace{15mm}
    \subfigure[]
    {\label{fig.tri-AD-par}\includegraphics[width=0.3\textwidth]{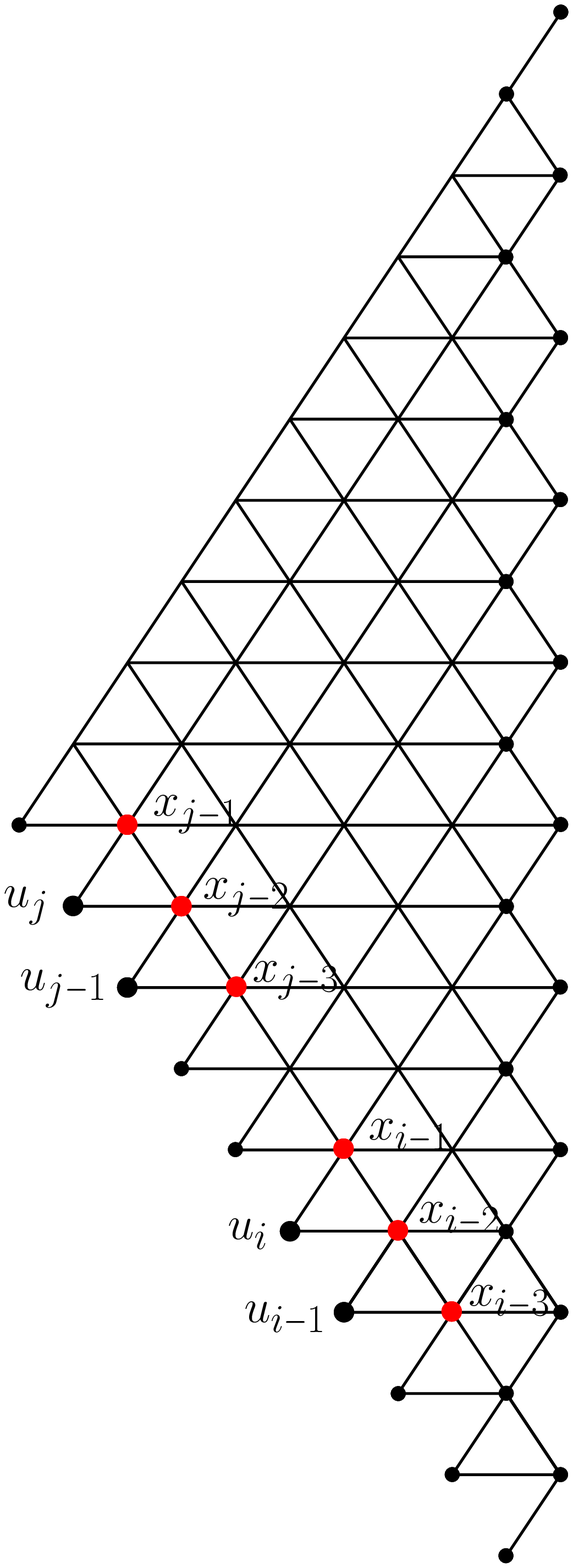}}
    \caption{(a) The graph $\mathcal{DS}(10)$. (b) An illustration of partitioning paths in \eqref{eq.Qpar3}.}
    \label{fig.tri2}
\end{figure}

\begin{proof}[Proof of Theorem \ref{thm.main2}]

We remind the reader that the entry $a_{i,j}$ of the matrix $A$ is given by $Q_{V^{*}}(u_i,u_j)$ on the graph $\mathcal{DS}(n)$. It suffices to show that $Q_{V^{*}}(u_i,u_j)$ satisfies the recurrence relation in \eqref{eq.rec}.

We will prove the recurrence relation by analyzing $Q_{V^{*}}(u_i,u_j)$ in three cases. Case $1$ and Case $2$ take care of two recursive expressions in \eqref{eq.rec} for $i>3$. Case $3$ deals with the initial condition $i=1$ in \eqref{eq.rec} and the two recursive expressions when $i=2$ and $i=3$.
\begin{itemize}
  \item[Case 1:] $i>3$ and $j>i+1$. \\
    By \eqref{eq.Q*}, \eqref{eq.partition2} and the linearity of determinants, we have the following identity (see Figure \ref{fig.tri-AD-par} for an illustration).
    \begin{equation}\label{eq.Qpar3}
      Q_{V^{*}}(u_{i},u_{j}) = Q_{V^{*}}(x_{i-1},x_{j-1}) + Q_{V^{*}}(x_{i-1},x_{j-2}) + Q_{V^{*}}(x_{i-2},x_{j-1}) + Q_{V^{*}}(x_{i-2},x_{j-2}).
    \end{equation}
    By Proposition \ref{prop.aux}, one can apply \eqref{eq.auxrec} to each terms on the right hand side of \eqref{eq.Qpar3}, this gives the expression of $Q_{V^{*}}(u_{i},u_{j})$ into the sum of sixteen terms.
    \begin{align}
      Q_{V^{*}}(u_i,u_j) & = Q_{V^{*}}(x_{i-2},x_{j-2}) + Q_{V^{*}}(x_{i-1},x_{j-2}) + Q_{V^{*}}(x_{i-2},x_{j-1}) + 2(-1)^{i-2} \nonumber\\
                         & + Q_{V^{*}}(x_{i-2},x_{j-3}) + Q_{V^{*}}(x_{i-1},x_{j-3}) + Q_{V^{*}}(x_{i-2},x_{j-2}) + 2(-1)^{i-2} \nonumber\\
                         & + Q_{V^{*}}(x_{i-3},x_{j-2}) + Q_{V^{*}}(x_{i-2},x_{j-2}) + Q_{V^{*}}(x_{i-3},x_{j-1}) + 2(-1)^{i-3} \nonumber\\
                         & + Q_{V^{*}}(x_{i-3},x_{j-3}) + Q_{V^{*}}(x_{i-2},x_{j-3}) + Q_{V^{*}}(x_{i-3},x_{j-2}) + 2(-1)^{i-3}. \label{eq.Q16}
    \end{align}

    One can view these sixteen terms displayed on the right hand side of \eqref{eq.Q16} as a $4 \times 4$ array. By \eqref{eq.Qpar3}, the sum of the four terms in the first, second and third columns can be simplified to $Q_{V^{*}}(u_{i-1},u_{j-1})$, $Q_{V^{*}}(u_{i},u_{j-1})$ and $Q_{V^{*}}(u_{i-1},u_{j})$, respectively. The four numbers in the last column add up to zero. Therefore, we obtain
    \begin{equation}\label{eq.Qrec2}
      Q_{V^{*}}(u_{i},u_{j}) = Q_{V^{*}}(u_{i-1},u_{j-1}) + Q_{V^{*}}(u_{i},u_{j-1}) + Q_{V^{*}}(u_{i-1},u_{j}).
    \end{equation}

  \item[Case 2:] $i>3$ and $j=i+1$. \\
    The term $Q_{V^{*}}(x_{i-1},x_{j-2})$ in \eqref{eq.Qpar3} vanishes, so we have
    \begin{equation}\label{eq.Qpar4}
      Q_{V^{*}}(u_{i},u_{i+1}) = Q_{V^{*}}(x_{i-1},x_{i}) + Q_{V^{*}}(x_{i-2},x_{i}) + Q_{V^{*}}(x_{i-2},x_{i-1}).
    \end{equation}
    We then apply \eqref{eq.auxrec} to the second term on the right hand side of \eqref{eq.Qpar4}, whereas we apply \eqref{eq.auxrec0} to the other terms, which gives an expression of $Q_{V^{*}}(u_{i},u_{i+1})$ into the sum of ten terms.
    \begin{align}
      Q_{V^{*}}(u_i,u_{i+1}) & = Q_{V^{*}}(x_{i-2},x_{i-1}) + Q_{V^{*}}(x_{i-2},x_{i})+ 2(-1)^{i-2} \nonumber\\
                         & + Q_{V^{*}}(x_{i-3},x_{i-1}) + Q_{V^{*}}(x_{i-2},x_{i-1}) + Q_{V^{*}}(x_{i-3},x_{i}) + 2(-1)^{i-3} \nonumber\\
                         & + Q_{V^{*}}(x_{i-3},x_{i-2}) + Q_{V^{*}}(x_{i-3},x_{i-1})+ 2(-1)^{i-3}. \label{eq.Q10}
    \end{align}

    By \eqref{eq.Qpar4}, the sum of three terms in the first column can be rewritten as $Q_{V^{*}}(u_{i-1},u_{i})$. The sum of the other four terms is $Q_{V^{*}}(u_{i-1},u_{i+1})$ by \eqref{eq.Qpar3}. The three remaining numbers add up to $2(-1)^{i-3} = 2(-1)^{i-1}$. Therefore, we have
    \begin{equation}\label{eq.Qrec4}
      Q_{V^{*}}(u_{i},u_{i+1}) = Q_{V^{*}}(u_{i-1},u_{i}) + Q_{V^{*}}(u_{i-1},u_{i+1}) + 2(-1)^{i-1}.
    \end{equation}

  \item[Case 3:] $i=1,2,3$ and $j>i$.\\
    We can compute $Q_{V^{*}}(u_i,u_j)$ directly from \eqref{eq.Q*} when $i=1$ and $i=2$. The numbers of paths going from $u_1,u_2$ and $u_j$ to any $v \in V$ are listed in Table \ref{tab.auxgv3}.
    \begin{table}[H]
    \centering
    \begin{tabular}{c|cc|cc|cc|c}
            & $v_1$ & $v_2$ & $v_3$ & $v_4$ & $v_5$ & $v_6$ & $\cdots$  \\
      \hline
      $u_1$ & 1 & 1 & 0 & 0 & 0 & 0 & $\cdots$  \\
      \hline
      $u_2$ & 0 & 2 & 1 & 1 & 0 & 0 & $\cdots$  \\
      \hline
      $u_j$ & 0 & $d_{j-1,0}+d_{j-2,0}$ & $d_{j-2,0}+d_{j-3,0}$ & $d_{j-2,1}+d_{j-3,1}$ & $d_{j-3,1}+d_{j-4,1}$ & $d_{j-3,2}+d_{j-4,2}$ & $\cdots$
    \end{tabular}
    \vspace{0.3cm}
    \caption{The numbers of Delannoy paths going from $u_1,u_2$ and $u_j$ to any $v \in V$ in the graph $\mathcal{DS}(n)$.}\label{tab.auxgv3}
    \end{table}

    If $i=1,j>1$, then we have
    \begin{equation}\label{eq.check1}
      Q_{V^{*}}(u_1,u_j) = \det \begin{bmatrix}
                                  1 & 1 \\
                                  0 & d_{j-1,0}+d_{j-2,0}
                                \end{bmatrix}
                         = 2.
    \end{equation}
    This shows the initial condition of the recurrence relation $Q_{V^{*}}(u_1,u_j) = 2$ for $j>1$.

    If $i=2,j>2$, then we have
    \begin{equation}\label{eq.check2}
      Q_{V^{*}}(u_2,u_j) = \det \begin{bmatrix}
                                  0 & 2 \\
                                  0 & d_{j-1,0}+d_{j-2,0}
                                \end{bmatrix}
                         + \det \begin{bmatrix}
                                  1 & 1 \\
                                  d_{j-2,0}+d_{j-3,0} & d_{j-2,1}+d_{j-3,1}
                                \end{bmatrix}
                         =4j-10.
    \end{equation}

    If $i=3, j>3$, then we are able to partition the paths going from $u_3$ into two sets of paths going from $x_1$ and $x_2$. By Lemma \ref{lemma.aux2}, we have
    \begin{align}
      Q_{V^{*}}(u_{3},u_{j}) & = Q_{V^{*}}(x_{2},x_{j-1}) + Q_{V^{*}}(x_{2},x_{j-2}) + Q_{V^{*}}(x_{1},x_{j-1}) + Q_{V^{*}}(x_{1},x_{j-2}) \nonumber\\
                         & = 2(j-3)(j-2) + 2(j-4)(j-3) + 2(j-2) + 2(j-3) \nonumber\\
                         & = 2(2j^2-10j+13). \label{eq.check3}
    \end{align}

    We can readily check that when $i=2$ and $i=3$, two recursive expressions in \eqref{eq.rec} hold by the straight forward computation using \eqref{eq.check1}, \eqref{eq.check2} and \eqref{eq.check3}.
    \begin{equation*}
        \begin{cases}
        Q_{V^{*}}(u_{i},u_{j}) = Q_{V^{*}}(u_{i-1},u_{j}) + Q_{V^{*}}(u_{i},u_{j-1}) + Q_{V^{*}}(u_{i-1},u_{j-1}), & j>i+1, \\
        Q_{V^{*}}(u_{i},u_{i+1}) = Q_{V^{*}}(u_{i-1},u_{i+1}) + Q_{V^{*}}(u_{i-1},u_{i}) + 2(-1)^{i-1}. &
        \end{cases} \tag{1.2}
    \end{equation*}
\end{itemize}
This completes the proof of Theorem \ref{thm.main2}.
\end{proof}

Besides the recursive way to find $a_{i,j}$, we are going to show that $a_{i,j}$ can be written explicitly as an alternating sum of entries $s_{p,q}$ in the Schr{\"o}der triangle.
\begin{equation*}
    a_{i,j} = 2 \sum_{\ell = 1}^{i}(-1)^{\ell-1}s_{i-\ell,j-\ell-1}. \tag{1.3}
\end{equation*}
We remind the reader that $s_{p,q}$ can be obtained recursively in \eqref{eq.recsch}.

\begin{proof}[Proof of Corollary \ref{cor.main}]
  It suffices to show that the right hand side of \eqref{eq.cor} satisfies the recurrence relation \eqref{eq.rec} stated in Theorem \ref{thm.main2}. We proceed with the following three cases.
  \begin{itemize}
    \item[Case $1$:] $i=1,j>1$.\\
        We use the fact that $s_{0,q} = 1$ for $q \geq 0$, then
        \begin{equation*}
          2 \sum_{\ell = 1}^{1}(-1)^{\ell-1}s_{1-\ell,j-\ell-1} = 2 (-1)^{1-1} s_{0,j-2} = 2.
        \end{equation*}

    \item [Case $2$:] $i>1, j > i+1$.\\
        We use the recurrence relation of $s_{p,q}$ given in \eqref{eq.recsch}, then we obtain
        \begin{align}
            2 \sum_{\ell = 1}^{i} (-1)^{\ell-1} s_{i-\ell,j-\ell-1} & = 2 \sum_{\ell = 1}^{i-1} (-1)^{\ell-1} s_{i-\ell,j-\ell-1} + 2 (-1)^{i-1} s_{0,j-i-1} \nonumber\\
            & = 2 \sum_{\ell = 1}^{i-1} (-1)^{\ell-1} \left( s_{i-\ell-1,j-\ell-1} + s_{i-\ell-1,j-\ell-2} + s_{i-\ell,j-\ell-2} \right) + 2 (-1)^{i-1} s_{0,j-i-1} \nonumber\\
            & = 2 \left( \sum_{\ell = 1}^{i-1} (-1)^{\ell-1} s_{i-\ell-1,j-\ell-1} \right) + 2 \left(\sum_{\ell = 1}^{i-1} (-1)^{\ell-1} s_{i-\ell-1,j-\ell-2} \right) + 2 \left(\sum_{\ell = 1}^{i} (-1)^{\ell-1} s_{i-\ell,j-\ell-2} \right). \nonumber
        \end{align}

    \item [Case $3$:] $i>1, j = i+1$.\\
        Similar to Case $2$, we have
        \begin{align}
            2 \sum_{\ell = 1}^{i} (-1)^{\ell-1} s_{i-\ell,i-\ell}  & = 2 \sum_{\ell = 1}^{i-1} (-1)^{\ell-1} s_{i-\ell,i-\ell} + 2 (-1)^{i-1} s_{0,0} \nonumber\\
            & = 2 \sum_{\ell = 1}^{i-1} (-1)^{\ell-1} \left( s_{i-\ell-1,i-\ell} + s_{i-\ell-1,i-\ell-1} \right) + 2 (-1)^{i-1} \nonumber\\
            & = 2 \left( \sum_{\ell = 1}^{i-1} (-1)^{\ell-1} s_{i-\ell-1,i-\ell} \right) + 2 \left(\sum_{\ell = 1}^{i-1} (-1)^{\ell-1} s_{i-\ell-1,i-\ell-1} \right) + 2 (-1)^{i-1}.  \nonumber
        \end{align}
  \end{itemize}
This completes the proof of Corollary \ref{cor.main}.
\end{proof}

\section{Open problems}\label{sec.open}

In this section, we formulate some open problems arising from enumerating the off-diagonal symmetry class of domino tilings of the Aztec diamond.

We first generalize our skew-symmetric matrix $A$ to $A(k,t)$, by introducing parameters $k$ and $t$ into its boundary conditions. The $(i,j)$-entry of $A(k,t)$ is defined by
\begin{equation}\label{eq.kt-matrix}
  \begin{cases}
    a_{1,j}  = t, & j > 1, \\
    a_{i,j}  = a_{i-1,j} + a_{i,j-1} + a_{i-1,j-1}, & j>i+1,\quad i>1, \\
    a_{i,j}  = a_{i-1,j} + a_{i-1,j-1} + k(-1)^{i-1}, & j=i+1,\quad i>1, \\
    a_{i,j}  = -a_{j,i}, & j \leq i.
\end{cases}
\end{equation}
Clearly, the original matrix $A$ is obtained from $A(k,t)$ with $k=t=2$.

By Conjecture \ref{conj} and Theorem \ref{thm.main1}, $\pf(A_{[2n]})$ could be expressed as the product of two consecutive terms of a sequence. The generalized matrix $A(k,t)$ defined above also seems to have this surprising property. We state this in the following conjecture. This conjecture has been checked by computer up to $n=25$.

%
\begin{conjecture}\label{conjext}
  Let $A(k,t)$ be the skew-symmetric matrix defined by \eqref{eq.kt-matrix}. Then we have
  \begin{equation}\label{eq.conj1}
    \pf(A_{[2n]}(k,t)) = t o_{n-1}(k,t)o_n(k,t),
  \end{equation}
 where $o_n(k,t) \in \mathbb{Z}[k,t]$, the polynomial ring over $\mathbb{Z}$ in two variables (see Table \ref{tab.o(k,t)} for the first nine terms of this sequence). In particular, Conjecture \ref{conj} is the special case when $k=t=2$.
\end{conjecture}
\begin{table}[htb!]
\centering
  \begin{tabular}{c|c}
    $n$ & $o_n(k,t)$ \\
    \hline
    $0$ & $1$ \\
    \hline
    $1$ & $1$ \\
    \hline
    $2$ & $-k+4t$ \\
    \hline
    $3$ & $-3k+16t$ \\
    \hline
    $4$ & $13k^2-120kt+256t^2$ \\
    \hline
    $5$ & $149k^2-1584kt+4096t^2$ \\
    \hline
    $6$ & $-2661k^3+38540k^2t-178688kt^2+262144t^3$ \\
    \hline
    $7$ & $-119335k^3+1899616k^2t-9887744kt^2+16777216t^3$ \\
    \hline
    $8$ & $8669753k^4-171171824k^3t+1234228224k^2t^2-3832545280kt^3+4294967296t^4$
  \end{tabular}
  \vspace{0.3cm}
  \caption{The first nine terms of $o_n(k,t)$.}\label{tab.o(k,t)}
\end{table}

The origin of the matrix $A$ comes from enumerating off-diagonally symmetric domino tilings of the Aztec diamond. It is natural to ask the following problem.
\begin{problem}
  Do the parameters $k$ and $t$ in the matrix $A(k,t)$ represent something naturally from the viewpoint of off-diagonally symmetric domino tilings of the Aztec diamond?
\end{problem}

Moreover, it would be interesting to find a combinatorial interpretation of Conjecture \ref{conj}. This leads to the following problem.
\begin{problem}
  Can we interpret off-diagonally symmetric domino tilings of the Aztec diamond of order $2n$ as a pair of some combinatorial objects?
\end{problem}

Finally, we would like to find a better way to express our Pfaffians.
\begin{problem}
  Is there a closed-form expression of $\pf(A_{[2n]}(k,t))$?
\end{problem}

\begin{appendices}
\section{Pfaffian calculations}\label{sec.eva}

In this appendix, we provide an idea to calculate the Pfaffian of the generalized matrix $A_I(k,t)$ for some set $I$ recursively. Evaluating the determinant of so-called \textit{Pascal-like} matrices (whose entries satisfy some specific $3$-term recurrence) has been discussed before; see for instance \cite{Bacher02}, \cite{Kra02} and \cite{MNN08}. Many useful and efficient tools to evaluate determinants are listed in the survey papers by Krattenthaler (\cite{Det1} and \cite{Det2}).

However, the recurrence satisfied by the entries of $A(k,t)$ is slightly different. The techniques presented in the papers mentioned in the previous paragraph seem not to work on evaluating the Pfaffian of our matrix $A_I(k,t)$ with $I$ given.

Our idea is motivated by the Pfaffian decomposition in the work of Ishikawa, Tagawa and Zeng \cite{ITZ13} which will be introduced in Section \ref{sec.pfdecomp}. The recursive calculation of $\pf(A_I(k,t))$ will be given in Section \ref{sec.pfeva}.

\subsection{The Pfaffian decomposition}\label{sec.pfdecomp}

The Pfaffian decomposition is the LDU-decomposition of a skew-symmetric matrix. We state this decomposition in Theorem \ref{thm.pfdecomp}, see \cite[Theorem 2.2]{ITZ13} for more details.

\begin{theorem}[Ishikawa, Tagawa and Zeng \cite{ITZ13}]\label{thm.pfdecomp}
  Let $M$ be a $2n \times 2n$ skew-symmetric matrix. If $\pf(M_{[2i]}) \neq 0$ for $1 \leq i \leq n$, then $M$ can uniquely be written as
  \begin{equation}\label{eq.pfd1}
    M = R^{\intercal}TR,
  \end{equation}
  where $T$ is a diagonal block matrix and $R$ is an upper triangular block matrix given by
  \begin{equation}\label{eq.pfd2}
    T= \begin{bmatrix}
         0 & t_1 & 0 & 0 & \cdots & 0 & 0 \\
         -t_1 & 0 & 0 & 0 & \cdots & 0 & 0 \\
         0 & 0 & 0 & t_2 & \cdots & 0 & 0 \\
         0 & 0 & -t_2 & 0 & \cdots & 0 & 0 \\
         \vdots & \vdots & \vdots & \vdots & \ddots & \vdots & \vdots \\
         0 & 0 & 0 & 0 & \cdots & 0 & t_n \\
         0 & 0 & 0 & 0 & \cdots & -t_n & 0 \\
       \end{bmatrix}, \quad
    R= \begin{bmatrix}
         0 & 1 & r_{1,3} & r_{1,4} & \cdots & r_{1,2n-1} & r_{1,2n} \\
         -1 & 0 & r_{2,3} & r_{2,4} & \cdots & r_{2,2n-1} & r_{2,2n} \\
         0 & 0 & 0 & 1 & \cdots & r_{3,2n-1} & r_{3,2n} \\
         0 & 0 & -1 & 0 & \cdots & r_{4,2n-1} & r_{4,2n} \\
         \vdots & \vdots & \vdots & \vdots & \ddots & \vdots & \vdots \\
         0 & 0 & 0 & 0 & \cdots & 0 & 1 \\
         0 & 0 & 0 & 0 & \cdots & -1 & 0 \\
       \end{bmatrix}
  \end{equation}
  with the nonzero entries of $T$ and $R$ being
  \begin{equation}\label{eq.pfd3}
    t_{\ell} = \frac{\pf(M_{[2\ell]})}{\pf(M_{[2\ell-2]})}, \quad r_{i,j} = \frac{\pf(M_{[2d-2] \cup \{i,j\}})}{\pf(M_{[2d]})},
  \end{equation}
  where $1 \leq \ell \leq n$, $1 \leq i < j \leq 2n$ and $d= \floor{\frac{i+1}{2}}$. In particular, $r_{i,i+1}=1$ if $i$ is odd.
\end{theorem}

Normally, it is very difficult to obtain a formula for the Pfaffian of a matrix from its decomposition unless there is a nice expression for the $t_{\ell}$'s and $r_{i,j}$'s. Our goal is to find a way to express the $t_{\ell}$'s and $r_{i,j}$'s from the Pfaffian decomposition of our matrix $A_{[2n]}(k,t)$. By the explicit expressions \eqref{eq.pfd3}, we are able to calculate $\pf(A_I(k,t))$ for sets $I = [2d-2] \cup \{i,j\}$, where $1 \leq i < j \leq 2n$ and $d= \floor{\frac{i+1}{2}}$.

Now, suppose we have matrices $T$ and $R$ of the form \eqref{eq.pfd2}. Our idea is to make the entries of the matrix $R^{\intercal}TR$ (viewing the $t_{\ell}$'s and $r_{i,j}$'s as variables) agree with the entries of the generalized matrix $A(k,t)$. Thanks to the recurrence relation \eqref{eq.kt-matrix} satisfied by the entries of $A(k,t)$, the entries of $R^{\intercal}TR$ must satisfy the same recurrence relation. As a consequence, one can express the $t_{\ell}$'s and $r_{i,j}$'s recursively. An illustration of this process for small $n$ will be presented in Section \ref{sec.pfeva}.

\subsection{Recursive calculation of $\pf(A_I(k,t))$}\label{sec.pfeva}

The key idea is to use that the entries of the matrix $R^{\intercal}TR$ (which we want to equal $A(k,t)$) satisfy \eqref{eq.kt-matrix}. In order to illustrate how to calculate Pfaffians recursively, we give the first few entries of the upper triangular part of the matrix $R^{\intercal}TR$ in \eqref{eq.VTV} and proceed by the following steps.
\begin{footnotesize}
\begin{equation}\label{eq.VTV}
    \begin{bmatrix}
      0 & t_1 & t_1r_{1,3} & t_1r_{1,4} & t_1r_{1,5} & t_1r_{1,6} & \cdots \\
      *  & 0   & t_1r_{2,3} & t_1r_{2,4} & t_1r_{2,5} & t_1r_{2,6}  & \cdots\\
      *  &  *   &     0      & t_2-t_1r_{1,4}r_{2,3}+t_1r_{1,3}r_{2,4} & t_2r_{3,5}-t_1r_{1,5}r_{2,3}+t_1r_{1,3}r_{2,5}  & t_2r_{3,6}-t_1r_{1,6}r_{2,3}+t_1r_{1,3}r_{2,6} & \cdots \\
      * &   *  &      *      & 0 & t_2r_{4,5}-t_1r_{1,5}r_{2,4}+t_1r_{1,4}r_{2,5} & t_2r_{4,6}-t_1r_{1,6}r_{2,4}+t_1r_{1,4}r_{2,6} & \cdots \\
      *  &  *   &      *      & * & 0 & \cdots & \cdots
    \end{bmatrix}
\end{equation}
\end{footnotesize}
\begin{enumerate}[itemsep=0pt,topsep=0pt,parsep=0pt]
  \item[Step 1:] From the $(1,2)$-entry, we obtain
                  \begin{equation*}
                    t_1 = t.
                  \end{equation*}
  \item[Step 2:] From the first row, we have $t_1r_{1,j} = t$ and thus
                  \begin{equation*}
                    r_{1,j} = 1, \quad j \geq 2.
                  \end{equation*}
  \item[Step 3:] From the $(2,3)$-entry, we have $t_1r_{2,3} = t_1 + t_1r_{1,3} - k$ and thus
                  \begin{equation*}
                    r_{2,3} = \frac{-k+2t}{t}.
                  \end{equation*}
  \item[Step 4:] From the first and second rows, we have $t_1r_{2,j} = t_1r_{2,j-1} + t_1r_{1,j} + t_1r_{1,j-1}$ for $j>3$ with the initial value $r_{2,3}$ given in the previous step. Therefore,
                  \begin{equation*}
                    r_{2,j} = \frac{-k + (-4+2j)t}{t}, \quad j \geq 3.
                  \end{equation*}
  \item[Step 5:] From the $(3,4)$-entry, we have $t_2 = t_1r_{2,3} + t_1r_{2,4} + k + t_1r_{1,4}r_{2,3} - t_1r_{1,3}r_{2,4}$. After simplifying, we have
                  \begin{equation*}
                    t_2 = -k+4t.
                  \end{equation*}
  \item[Step 6:] From the second and third rows, after simplifying, we obtain $\displaystyle r_{3,j} = r_{3,j-1} + \frac{2t_1r_{2,j-1}}{t_2}$ for $j>4$ with the initial value $r_{3,4}=1$. Therefore,
                  \begin{equation*}
                    r_{3,j} = \frac{(7-2j)k+(12-10j+2j^2)t}{-k+4t}, \quad j \geq 4.
                  \end{equation*}
  \item[Step 7:] From the $(4,5)$-entry, we have $t_2r_{4,5} = (t_2-t_1r_{1,4}r_{2,3}+t_1r_{1,3}r_{2,4}) + (t_2r_{3,5}-t_1r_{1,5}r_{2,3}+t_1r_{1,3}r_{2,5}) - k + t_1r_{1,5}r_{2,4} - t_1r_{1,4}r_{2,5}$. After simplifying, we obtain
                  \begin{equation*}
                    r_{4,5} = 5.
                  \end{equation*}
  \item[Step 8:] From the third and fourth rows, after simplifying, we obtain
                  \begin{equation*}
                    r_{4,j} = r_{4,j-1} + \frac{-2t_1r_{2,3} + 2t_1r_{2,j-1} + t_2r_{3,j-1} + t_2r_{3,j}}{t_2}, \quad j > 5,
                  \end{equation*}
  with the initial value $r_{4,5}$ given in the previous step. Therefore,
                  \begin{equation*}
                    r_{4,j} = \frac{(75-42j+6j^2)k + (-32j+24j^2-4j^3)t}{3(-k+4t)}, \quad j \geq 5.
                  \end{equation*}
  \item[Step 9:] We can find $t_3$ from the $(5,6)$-entry, and so on.
\end{enumerate}

In general, using the recurrence relation \eqref{eq.kt-matrix}, the $(2i-1,2i)$-entry of \eqref{eq.VTV} gives the formula for $t_i$ while the $(2i,2i+1)$-entry gives the formula for $r_{2i,2i+1}$. The relation between rows $i-1$ and $i$ gives a first order non-homogeneous recurrence relation $r_{i,j} = r_{i,j-1} + \text{some lower terms}$ for $j>i+1$, with the initial value $r_{i,i+1} = 1$ if $i$ is odd, and $r_{i,i+1}$ given in the previous step if $i$ is even.

By two equations in \eqref{eq.pfd3}, one obtains
\begin{equation}\label{eq.pfA}
  \pf(A_{[2n]}(k,t)) = \prod_{\ell=1}^{n}t_{\ell},
\end{equation}
and if $I = [2d-2] \cup \{i,j\}$, where $1 \leq i < j \leq 2n$ and $d= \floor{\frac{i+1}{2}}$, then we have
\begin{equation}\label{eq.pfAI}
 \pf(A_{I}(k,t)) = \pf(A_{[2d]}(k,t))r_{i,j}.
\end{equation}

The formulas for the first eight terms of $t_{\ell}$ are listed below. Equation \eqref{eq.pfA} and these formulas verify Conjecture \ref{conjext} for $n \leq 8$.
\begin{align}
 t_1 & = t. \nonumber \\
 t_2 & = -k+4t. \nonumber \\
 t_3 & = -3k+16t. \nonumber \\
 t_4 & = \frac{13k^2-120kt+256t^2}{-k+4t}. \nonumber \\
 t_5 & = \frac{149k^2-1584kt+4096t^2}{-3k+16t}. \nonumber \\
 t_6 & = \frac{-2661k^3+38540k^2t-178688kt^2+262144t^3}{13k^2-120kt+256t^2}. \nonumber \\
 t_7 & = \frac{-119335k^3+1899616k^2t-9887744kt^2+16777216t^3}{149k^2-1584kt+4096t^2}. \nonumber \\
 t_8 & = \frac{8669753k^4-171171824k^3t+1234228224k^2t^2-3832545280kt^3+4294967296t^4}{-2661k^3+38540k^2t-178688kt^2+262144t^3}. \label{eq.ts}
\end{align}

\begin{remark}
  Following the above process, one can write down the general forms of $t_{\ell}$ and $r_{i,j}$ recursively. For example, when $\ell \geq 3$ we have
  \begin{equation}\label{eq.tl}
    t_{\ell} = f_1(2\ell-1) + \sum_{i=2}^{\ell-2}f_2(i,2\ell) + f_3(\ell-1,2\ell),
  \end{equation}
  where the functions $f_1,f_2$ and $f_3$ are defined as follows.
  \begin{align*}
    f_1(y)      & = 2t_1(r_{2,y}-r_{2,y-2}). \\
    f_2(x,y)    & = t_{x}\big( (r_{2x-1,y}r_{2x,y-1}-r_{2x-1,y-1}r_{2x,y}) +(r_{2x-1,y-2}r_{2x,y}-r_{2x-1,y}r_{2x,y-2}) \\
                & + (r_{2x-1,y-3}r_{2x,y-1}-r_{2x-1,y-1}r_{2x,y-3}) + (r_{2x-1,y-3}r_{2x,y-2}-r_{2x-1,y-2}r_{2x,y-3}) \big).\\
    f_3(x,y)    & = t_{x} \big( (r_{2x-1,y}r_{2x,y-1}-r_{2x-1,y-1}r_{2x,y}) +r_{2x-1,y-2} + r_{2x-1,y-1} + r_{2x,y} \big).
  \end{align*}

  However, the general forms that we obtained in this way are very complicated, and they do not seem to be helpful for proving Conjecture \ref{conjext}. 
\end{remark}

\end{appendices}

\subsection*{Acknowledgements}

The author thanks Mihai Ciucu for stimulating discussions and helpful suggestions on the preliminary version of this paper. The author also thanks the reviewers for carefully reading the manuscript and providing helpful comments.




\end{document}